\documentclass[a4paper,11pt]{amsart}
\usepackage[utf8]{inputenc}

\usepackage{color}
\usepackage[english]{babel}
\usepackage{amsmath,amsthm}
\usepackage{amsfonts}
\usepackage{graphicx}
\usepackage[all]{xy}
\usepackage{appendix}
\usepackage{centernot}
\usepackage{amssymb}
\DeclareGraphicsExtensions{.pdf,.png,.jpg}
\usepackage{xcolor}

\newtheorem{thm}{Theorem}[section]

\newtheorem{lem}[thm]{Lemma}
\newtheorem{prop}[thm]{Proposition}

\newtheorem{Question}{Question}
\newtheorem*{lemma}{Lemma}

\usepackage{varioref}
\theoremstyle{definition}
\newtheorem*{question}{Question}

\newtheorem{defn}[thm]{Definition}
\newtheorem{rem}[thm]{Remark}

\numberwithin{equation}{section}
\title[A Frobenius Theorem for Continuous Distributions]{A Frobenius Theorem for Corank-1 Continuous Distributions   in Dimensions two and three
}

\author{Stefano Luzzatto}
\author{Sina Tureli}
\author{Khadim War}

\thanks{\emph{Acknowledgements:} We would like to thank Ra\'ul Ures and Jana Rodriguez-Hertz for their lectures at ICTP and many useful conversations, and in particular a question of Ra\'ul Ures which partly motivated the strategy in this paper.  We also thank Giovanni Bellettini for some interesting and stimulating comments during the initial part of the research. We are also very thankful to Seppo Hiltunen fand to the referee for their thorough reading of a preliminary version of this paper and comments which helped us to improve the argument and the presentation in several places. This research has been supported [in part] by EU Marie-Curie IRSES Brazilian-European partnership in Dynamical Systems (FP7-PEOPLE-2012-IRSES 318999 BREUDS). 
}
\date{18 March 2016}

\begin{document}

\maketitle

\begin{abstract}
We formulate a notion  of    \emph{(uniform) asymptotic involutivity} and show that it implies (unique) integrability of corank-1 continuous distributions in dimensions three or less. This generalizes and extends a classical theorem  of  Frobenius Theorem which says that an involutive  \(  C^{1}  \) distribution is uniquely integrable.
 \end{abstract}

\section{Introduction and statement of results}
 Let \(  M  \) be a smooth Riemannian manifold of dimension \(  n\geq 2  \). A \emph{rank-$k$ distribution} $\Delta$ on \(  M  \) is a choice of $k$-dimensional linear subspaces $\Delta_p\subset T_pM$ at each point $p\in M$. A (local) integral manifold \(  N  \) of $\Delta$ is a
 submanifold $N\subset M$ such that  $T_pN=\Delta_p$ at each point $p\in N$. The distribution $\Delta$ is
 \emph{integrable}  if there exists an integral manifold through every point, and \emph{uniquely integrable} if this
 integral manifold is unique.

The integrability and unique integrability of a given distribution are classical questions which generalize the problem of existence and uniqueness of solutions for ODE's. Unlike in the case of ODE's, there are counterexamples to show that for higher-dimensional distributions smoothness plays a role but is, in general, not sufficient to guarantee integrability. Indeed, in the higher-dimensional setting, some \emph{involutivity} conditions are required over and above some regularity conditions such as \(  C^{1}  \) or Lipschitz, and indeed classical involutivity conditions require some regularity to even be formulated.  One of the main points of this paper is to formulate some involutivity conditions for \emph{continuous} corank-1 distributions and to show that they imply integrability and, in some cases, unique integrability. 

For starters we recall classical results for ODE's which imply that any rank one continuous distribution is integrable but possibly non-uniquely. Therefore our focus for the integrability part will be on rank two continuous distributions inside three dimensional manifolds, while the uniqueness results will also focus on rank-1 distributions on surfaces. We will formulate these conditions in two stages: the first one more natural and the second one more general and more technical and more useful for applications. Since integrability is a local property we will  work in some fixed local chart of the manifold; in particular the corank-1 distribution in this local chart can always be written as the kernel of some 1-form \(  \eta  \). If \(  \eta  \) is sufficiently regular (e.g. \(  C^{1}  \)), it admits an exterior derivative which we will denote by \(  d\eta  \). All norms to be used below will be the norms induced by the Riemannian volume. Unless specified otherwise all norms and converging sequences refer to the  \(  C^{0}  \) topology.

\subsection{Asymptotic involutivity}

\begin{defn}\label{defasyminv} A continuous corank-1 distribution \(  \Delta = ker (\eta) \) is \emph{asymptotically involutive} if there exists a  sequence of \(  C^{1}  \) differential 1-forms \(  \eta_{k}  \) with \( \eta_{k}\to \eta \) such that \[ \|\eta_{k}\wedge d\eta_{k}\|e^{\|d\eta_k\|} \to 0 \] as \(  k\to\infty  \).  \(  \Delta  \) is \emph{uniformly asymptotically involutive} if moreover we have \[ \|\eta_{k}-\eta\| e^{\|d\eta_k\|} \to 0. \] \end{defn} 

\begin{thm}\label{Frobenius} Let \(  \Delta   \) be a corank-1 distribution on a smooth manifold of dimension 2 or 3. If \(  \Delta  \) is asymptotically involutive then it is integrable. If $\Delta$ is uniformly asymptotically involutive then it is uniquely integrable.
\end{thm}

Notice  that the asymptotic involutivity condition does not apply to rank-1 distributions (or rather, is trivially satisfied in that case). Indeed, continuous rank-1 distributions are always integrable and in this case the relevant statement just concerns the uniqueness part.

The classical Frobenius Theorem \cite{Fro}  yields unique integrability for \(  C^{1}  \)  distributions (in arbitrary dimension)  under the assumption that \(  \Delta  \) is \emph{involutive}: \begin{equation}\label{eq:involutive} \eta\wedge d\eta=0. \end{equation}
 This can be seen as a special case of Theorem \ref{Frobenius} by choosing  \(  \eta_{k}\equiv \eta.  \)
Other proofs, generalizations and extensions of Frobenius' Theorem exist in the literature \cite{CenLacVer93, Han09, Han09b, CM, Hil00,   Lee, L, MonMor13, Sni07, Sus73}, related both to  the regularity of the distribution  and to the setting of the problem,
 including  generalizations to Lipschitz distributions \cite{MM, Ram07, S},  for which condition \eqref{eq:involutive} can  be formulated
 almost everywhere, and the interesting, though apparently not very well known, generalization of Hartman \cite{H, H2}
 to \emph{weakly differentiable} distributions, i.e.
distributions defined by a 1-form \(  \eta  \) which may not be differentiable or even Lipschitz but still admits
 continuous exterior derivative\footnote{More precisely we say that  \(  \eta  \)
is ``weakly differentiable'' if there exists a differential 2-form \(  d\eta  \) that satisfies Stokes' Formula: \( \int _{J} \eta = \int\int_{S} d\eta \) for every piece of \(  C^{1}  \) surface \(  \mathcal S  \)  bounded by a \(  C^{1}  \) piecewise Jordan curve \(  J  \). Note that this condition holds for example under the assumption that \(  \eta  \) is Lipschitz (and therefore differentiable almost everywhere) and is therefore strictly weaker than assuming that \(  \eta  \) is \(  C^{1}  \).} \(  d\eta  \)  and for which,  therefore, condition \eqref{eq:involutive}  can also still be formulated.

Our definition of asymptotic involutivity  allows for a significant relaxation of the assumptions on the regularity of \(  \Delta  \) and in particular  \emph{does not require} that \(  \eta  \) admit a continuous exterior derivative \( d\eta \). 
 Indeed, necessary and sufficient conditions for the existence of \( d\eta \) is given by the following Lemma due to Hartman. 

 \begin{lemma}[\cite{H}]
A continuous differential 1-form \( \eta \) admits a continuous exterior derivative \(  d\eta  \) if and only if there exists  a sequence of \(  C^{1}  \) differential 1-forms \(  \eta_{k}  \) such that
 \(   \eta_{k}\to \eta \) and \(   d\eta_{k} \to d\eta  \). 
 \end{lemma}

Our conditions  \( \|\eta_k\wedge d\eta_k\|e^{\|d\eta_k\|}\to 0  \) and \(
 \|\eta_{k}-\eta\| e^{\|d\eta_k\|} \to 0  \) \emph{do not require} \( d\eta_{k} \) to converge
and are thus strictly weaker than any condition which requires the existence of a continuous exterior derivative. In Remark \ref{rem:pde} we give an example where the approximating exterior derivatives cannot converge but where our conditions are satisfied.

One of the main reasons why we are restricted to lower dimensions and codimensions in this
paper is that the techniques we use requires us to locally solve some linear PDE`s.
In case of higher dimensions these become systems of linear PDEs where a solution is not generally possible.
We believe that some generalization of these result to higher dimensions and codimensions should be possible but most likely using completely different techniques.

\begin{rem}
An alternative, and fairly standard, formulation of the Theorem of Frobenius is the following: \emph{a \( C^{1} \) distribution is integrable if and only if it can be described by vector fields whose flows commute}.  One of the difficulties in generalizing this statement to our setting is that, for continuous distributions, integrability and unique integrability become distinct issues, whereas in the \( C^{1} \) setting, integrability automatically implies unique integrability. As far as integrability is concerned, this alternative formulation of Frobenius' Theorem does come into play in our proofs. Indeed, we consider the \( C^{1} \) approximating distributions defined by the approximating \( C^{1} \) forms \( \eta_{k} \) and use flows associated to vector fields which span these distributions to construct approximating surfaces and our main technical result, Theorem \ref{solution} below, is a quantitative estimate of how ``non-involutive'' these surfaces are, or equivalently, how close the vector fields are from being commutative. For the uniqueness part, we do use a form of this alternative formulation: in Proposition \ref{thm:uniqueness} we show that our distribution is locally spanned by two uniquely integrable vector fields and then use this fact to deduce the unique integrability of the distribution. 
\end{rem}

 \subsection{Asymptotic involutivity on average} \label{sec:ave}
 A key feature of the asymptotic involutivity condition  is that it allows   \(  \|d\eta_{k}\|  \) to blow up, as
 should be expected to happen in the case of the approximation of distributions which are not Lipschitz, albeit at
 some controlled rate. Our argument however yields some rather technical, but significantly more general, conditions
 which relax to some extent the requirement on the rate at which \(  \|d\eta_{k}\|  \) is allowed to blow up, and
 instead  only require some control on the rate at which the ``average'' value of \(  \|d\eta_{k}\|  \) blows up. For the applications that we have in mind, in particular dynamical systems, these kind of conditions will be more possible to verify as opposed to conditions given in the last section.

In the case we are working with a rank $2$ distribution on a three dimensional manifold, we fix some arbitrary point \(  x_{0}\in M  \) and  a local coordinate system $(x^1, x^{2}, x^3, \mathcal{U})$ around $x_0$. We suppose we are given a continuous form \(  \eta  \)  defined in \(  \mathcal U  \) and the corresponding distribution \(  \Delta=ker(\eta)  \), and assume without loss of generality that \(  \Delta  \) is everywhere transversal to  the coordinate axis  \(  \partial/\partial x^{3}  \). For any sequence of \(  C^{1}  \) forms  \(   \eta_{k}  \) defined in \(  \mathcal U  \) we write the corresponding exterior derivative \(  d\eta_{k}  \)  in coordinates as
 $$
 d\eta_k=d\eta_{k,1}dx^1\wedge dx^3+d\eta_{k,2}dx^2\wedge dx^3+d\eta_{k,3}dx^1\wedge dx^2
 $$
 where \(  d\eta_{k, 1}, d\eta_{k, 2}, d\eta_{k,3}  \) are \(  C^{1}  \) functions defined in \(  \mathcal U  \).
If \(  \eta_{k}\to \eta  \) then, for all \(  k  \) sufficiently large, the corresponding distributions \(  \Delta_{k}=ker (\eta_{k})  \) are also transversal to the coordinate axis \(  \partial/\partial x^{3}  \) and therefore there exist \(  C^{1}  \) frames \(  \{X_{k}, Y_{k}\}  \) for \(  \Delta_{k}  \) in \(  \mathcal U  \) where \(  X_{k}, Y_{k}  \) are \(  C^{1}  \) vector fields of the form \begin{equation}\label{XkYk} X_k=\frac\partial{\partial x^1}+a_k\frac{\partial}{\partial x^3}, \qquad Y_k=\frac{\partial}{\partial x^2}+b_k\frac{\partial}{\partial x^3} \end{equation}
 for some $C^1$ functions $a_k, b_k$. We let $e^{\tau X_k}, e^{\tau Y_k}$ denote the flows induced by the vector
 fields  $X_k, Y_k$ respectively. Fixing some smaller neighbourhood \(  \mathcal U'\subset \mathcal U  \) we can
 choose \(  t_{0}>0  \) such that  the flow is well defined and \(e^{\tau X_k}(x), e^{\tau Y_k}(x)\in \mathcal U  \)
 for all \(  x\in \mathcal U'  \) and  \(  |\tau|\leq t_{0}  \). Then, for every \(  x\in \mathcal U', |t|\leq t_{0}
 \)  we define
 \[
\widetilde{d\eta}_{k,1}(x,t):= \int_0^{t} d\eta_{k,1}\circ e^{\tau X^{(k)}}(x)d\tau, \] \[ \widetilde{d\eta}_{k,2}(x,t):= \int_0^{t} d\eta_{k,2}\circ e^{\tau Y^{(k)}}(x) d\tau \] and \[ \widetilde{d\eta}_k(x,t):=\max \{\widetilde{d\eta}_{k,1}(x,t), \widetilde{d\eta}_{k,2}(x,t)\}. \]

In the case when we have a rank $1$ distribution on a two dimensional manifold then we have coordinates $(x^1,x^2)$
a single vector-field $X^k$ spanning the distribution and only  $\widetilde{d\eta}_k(x,t):=\widetilde{d\eta}_{k,1}(x,t)$ (since there is no $Y^k$ flow, there is no $\widetilde{d\eta}_{k,2}(x,t)$). Everything else remains the same.

 \begin{defn}\label{defk}
 A continuous distribution $\Delta=ker(\eta)$   is \textit{asymptotically involutive on average} if, for every $x_0\in
 M$,
 there exist local coordinates around \(  x_{0}  \) and a sequence of $C^1$ differential 1-forms $\eta_k$ with \(
 \eta_{k}\to \eta  \) and corresponding \(  C^{1}  \) distributions \(  \Delta_{k}=ker(\eta_{k})  \) and \(  C^{1}  \)
 local frames \(  \{X_{k}, Y_{k}\}  \) (or just \(  \{X_{k}\}  \)), and a neighbourhood \(  \mathcal U'\subset \mathcal U  \) such that for every
 \(  x\in \mathcal U'  \) and every \( |t|\leq t_{0} \)
 \[
\|\eta_k\wedge d\eta_k\|_{x}e^{\widetilde{d\eta}_{k}(x,t)}\to 0 \]
  as \(  k\to \infty  \). \(  \Delta  \) is \emph{uniformly asymptotically involutive on average} if, moreover, for
  every \(  x\in \mathcal U'  \) and every \( |t|\leq t_{0} \),
\[ \|\eta_k-\eta\|_{x} e^{\widetilde{d\eta}_{k}(x,t)} \to 0 \]
 as \(  k\to \infty  \).
\end{defn} 

\begin{thm}\label{Frobenius-weak} Let \(  \Delta   \) be a corank 1 distribution on a manifold of dimension 2 or $3$. If \(  \Delta  \) is asymptotically involutive on average then it is integrable. If $\Delta$ is  uniformly asymptotically involutive on average
 then it is uniquely integrable.
 \end{thm}

Notice that \(  \widetilde{d\eta}_{k}(x,t) \leq \|d\eta_{k}\| \) and therefore Theorem \ref{Frobenius} follows immediately from Theorem \ref{Frobenius-weak}. 

We conclude this section with a question motivated by the observation that in the \(  C^{1}  \) setting the involutivity condition \eqref{eq:involutive} is both necessary and sufficient for unique integrability. It seems natural to ask whether the same is true for  uniform asymptotic involutivity on average.

\begin{Question} Let \(  \Delta  \) be a 2-dimensional continuous uniquely integrable distribution on a 3-dimensional  manifold. Is \(  \Delta  \)  uniformly  asymptotically involutive on average? \end{Question}

\subsection{Applications} We discuss here three applications of our results: to the problem of the uniqueness of solutions of ODE's, of existence and uniqueness of solutions of PDE's, and to the problem of integrability of invariant bundles in Dynamical Systems. While none of these applications perhaps has the status of a major result in itself, we believe they are good ``examples'' and indicate the potential applicability of our main integrability results to a wide range of problems in different areas of mathematics. The ODE and PDE applications follow from Theorem \ref{Frobenius} while the application to Dynamical Systems require the greater generality of Theorem \ref{Frobenius-weak}.

\subsubsection{Uniqueness of solutions for ODE's} We consider a vector field \begin{equation}\label{ODE} X=f(x) \end{equation} defined in some local chart \( \mathcal U \) of a two-dimensional Riemannian manifold \( M \) by a non-vanishing continuous function \( f \). By a classical result of Peano, \(  X  \) admits locally defined  integral curves at every point in \( \mathcal U \) but  uniqueness is not guaranteed as there exist simple counterexamples even if \( f \) is H\"older continuous. A natural question concerns the ``weakest'' form of continuity which guarantees uniqueness. We recall that the \emph{modulus of continuity} of a  continuous function \( f \) defined on \( \mathcal U \) is a continuous function \( w: [0, \infty) \to [0, \infty) \) such that \( w(t)\to 0 \) as \( t\to 0 \) and, for all \( x,y \in \mathcal U \), \[ |f(x)-f(y)| \leq w(|x-y|). \] %If \( w(t) = Kt \) then \( f \) is Lipschitz continuous and if \( w(t)=Kt^{\alpha} \) for some  \( \alpha\in (0,1)\) then \( f \) is H\"older continuous. As a Corollary of our arguments we obtain the following result which we will prove in Section \ref{ODEproof}.

\begin{thm}\label{thm-ODE} Suppose the modulus of continuity of \( f \) satisfies \begin{equation}\label{ODEstar} \lim_{\epsilon \rightarrow 0}\frac 1\epsilon \int_{0}^{\epsilon}\omega(t) dt \cdot \exp\left(\frac 1{\epsilon^{2}}\int_{0}^{\epsilon}\omega(t)dt \right)=0. \end{equation} Then  \( X \) has a unique local integral curve through every point  in $\mathcal  U$. \end{thm} It would be interesting to compare our condition \eqref{ODEstar} with other conditions  such as the classical and the well-known \emph{Osgood} condition \begin{equation}\label{osgood} \int_{0}^{\epsilon} \frac{1}{w(t)}dt = \infty \end{equation} which also  implies unique integrability \cite{Osg98} . For the moment however we have not been able to establish a relationship between the two conditions. \begin{question} Does either of \eqref{ODEstar} or \eqref{osgood} imply the other? Are there examples of functions which satisfy one and not the other? \end{question} We remark that any function which admits a modulus of continuity also admits an \emph{increasing} modulus of continuity \(  \hat w(t) \geq \omega(t)  \). Then if \eqref{ODEstar} holds for the modulus of continuity \(  \omega(t)  \) it clearly holds also for the increasing modulus of continuity \(  \hat\omega(t)  \) (but notice that the converse is not true). If the modulus of continuity is increasing, then   \eqref{ODEstar} is equivalent to the following more natural-looking condition \begin{equation}\label{ODEstarb}
  \lim_{t \rightarrow 0}\omega(t)e^{\frac{\omega(t)}{t}}=0.
  \end{equation}
It is easy to check that Lipschitz functions satisfy \eqref{ODEstarb} (and therefore also \eqref{ODEstar}) as well as \eqref{osgood},  as do some standard non-Lipschitz functions such as \( w(t) = t\ln t, w(t) =t\ln...\ln t \) and more ``exotic'' examples such as \(  w(t) = t\ln^{1+t}t  \),  whereas functions such as \(  w(t)=t^{\alpha}, \alpha\in (0,1)  \) and \(  w(t) = t\ln^{\alpha} t, \alpha>1  \) satisfy neither our condition \eqref{ODEstar} nor \eqref{osgood}.  It seems likely that if there exists any example of a modulus of continuity which satisfies \eqref{ODEstar} but not \eqref{osgood} it would have to have some significant amount of oscillation which could be controlled by the integrals in \eqref{ODEstar} but not by the simpler condition \eqref{ODEstarb}.

\begin{rem}
We remark that the result of Hartman in \cite{H} does not immediately give uniqueness for such vector-fields. Indeed consider the vector-field $f(x,y)=\frac{\partial}{\partial y} + xln(x)\frac{\partial}{\partial x}$ on $U \subset \mathbb{R}^2$. Then it is nullifier
is given by $\eta = xln(x)dy - dx$. If one approximates this 1-form by $C^1$ forms $\eta_k = f^k(x)dy -dx$
then $d\eta_k = \frac{df^k}{dx}dx\wedge dy$. Yet it is not possible for $d\eta_k$ to converge to a continuous differential form since $|\frac{df^k}{dx}|$ is not even bounded as $f^k(x)$ converges to a non-Lipschitz function.
\end{rem}

\subsubsection{Pfaff Equations} Besides the intrinsic interest of the question of integrability from a purely geometric point of view, the issue  of integrability classically arises in the context of the problem of existence and uniqueness of solutions of PDE's. Indeed, this seems to have been the main motivation of  Frobenius \cite{Fro}, who  applied previous results of Clebsch \cite{Cle} and Deahna \cite{Dea}, see discussion in \cite{Law}, to Pfaff equations, i.e. equations of the form \begin{equation}\label{Fro}\tag{\ensuremath{\mathcal P}} \begin{cases}
 \displaystyle{
 \frac{\partial f}{\partial x}(x,y)=a(x,y,f(x,y))
 }
 \\[1em]
 \displaystyle{
 \frac{\partial f}{\partial y}(x,y)=b(x,y,f(x,y))
} \end{cases} \end{equation} where $a(x,y,z), b(x,y,z)$ are   scalar functions defined on $\mathcal U = \mathcal{V}\times I\subset\mathbb{R}^3$. When $f=f(x,y)$ exists (and is unique), with the initial condition \(f(x_0,y_0)=z_0 \), the system \eqref{Fro} is said to be \emph{(uniquely) integrable} at $(x_0,y_0,z_0)$.

 The existence and uniqueness of the Pfaff system of equations clearly depends on the properties of the functions \(
 a  \) and \(  b  \). The classical Theorem of Frobenius gives some \emph{involutivity} conditions which imply
 integrability if the functions \(  a, b  \) are \(  C^{1}  \). As  a relatively straightforward application of our
 more general result, we can
 consider the situation where \(  a, b  \) are just continuous but have a particular, though not very restrictive,
 form. More specifically suppose that \(  a, b  \) have the form
 \begin{equation}\label{coeff}\tag{\protect\(  \tilde{\mathcal P}  \protect\)}
a (x, y, z):= A(x,y) F(z) \quad \text{ and }\quad b (x, y, z):= B(x,y) F(z) \end{equation} for  continuous functions  \(  A(x,y), B(x,y),   F(z)  \) satisfying: \begin{enumerate} \item[(\protect\( \tilde{\mathcal P}_{1})\protect\)] \(  F  \) is Lipschitz continuous \item[(\protect\( \tilde{\mathcal P}_{2})\protect\)] There exist sequences \(  A^{(k)}, B^{(k)}  \) of \(  C^{1}  \)
    functions such that
\begin{enumerate} \item[\protect\( i \protect\))] \(  A^{(k)}\to A  \) and \(  B^{(k)}\to B  \), \item[\protect\( ii \protect\))] \(  A^{(k)}_{y}-B^{(k)}_{x}\to 0  \). \end{enumerate} \end{enumerate} Note that \(  A^{(k)}_{y}, B^{(k)}_{x}  \) denote the partial derivatives of \(  A^{(k)}, B^{(k)}  \) with respect to \(  x  \) and \(  y  \) respectively, and that the convergence in \(  i)  \) and \(  ii)   \) of \(  (\tilde{\mathcal P}_{2})  \) are intended in the \(  C^{0}  \) topology.

\begin{thm}\label{pfaff} The Pfaff system \eqref{Fro} defined by functions of the form \eqref{coeff} satisfying \(  (\tilde{\mathcal P}_{1}), (\tilde{\mathcal P}_{2})  \) is uniquely integrable. \end{thm}

We will prove Theorem \ref{pfaff} in Section \ref{sec:pfaff}. We note that this is an application of our theorem \ref{Frobenius} in the view that integrating a rank $n$ distribution in a $n+m$ dimensional manifold is locally solving a system of $n$ linear PDE`s in $\mathbb{R}^{n+m}$. Thus theorem \ref{pfaff} will be a local application of our theorem
\ref{Frobenius}.

\begin{rem}\label{rem:pde}
 We remark that condition  \(  (\tilde{\mathcal P}_{2})  \) seems relatively abstract but it is quite easy to construct examples of continuous functions \(  A, B  \) which satisfy it. Suppose for example that \(  \tilde A(x), \tilde B(y)  \) are continuous functions and that \(  \varphi(x,y)  \) is a \(  C^{2}  \) function, and let \(  A(x,y) = \tilde A(x) + \varphi_{x}(x,y)  \) and \(  B(x,y) = \tilde B(y) + \varphi_{y}(x,y)  \). Let \(  \tilde A^{(k)}, \tilde B^{(k)}  \) be sequences of \(  C^{1}  \) functions with \( \tilde A^{(k)} \to \tilde A, \tilde B^{(k)} \to B   \). Then  \( A^{(k)}(x,y) = \tilde A^{(k)}(x) +  \varphi_{x}(x,y)   \) and \(B^{(k)}(x,y) = \tilde B^{(k)}(y) +  \varphi_{x}(x,y)     \) are  \(  C^{1}  \) functions and it follows that \(  A^{(k)}\to A  \), \(  B^{(k)}\to B  \). Moreover the partial derivatives are  \(  A^{(k)}_{y}= \varphi_{xy}\), \(B^{(k)}_{x}=\varphi_{yx} \) and therefore  \(  A^{(k)}_{y}-B^{(k)}_{x}   =0\).  A particular example for instance can be given by letting 
 \[
 A(x,y) = x^{\alpha} \qquad B(x,y) = y^{\alpha}, \qquad F(z) = z^{2}
 \]
% $$
%\eta = dz - x^{\alpha}z^2dx - y^{\alpha}z^2dy,
%$$
with $\alpha, \beta<1$. This   satisfies the conditions given above and thus defines a uniquely integrable PDE, but it is not possible to deduce it from theorems on Lipschitz distributions such as \cite{Ram07, S} since it clearly is not a Lipschitz distribution.
\end{rem}
%\footnote{ %The strategy taken by both Frobenius and Hartman is to % further reduce the problem to alternative systems of differential equations and discuss the equivalence between such systems. In the \(  C^{1}  \) case the equivalence between these problems is clear. % Under weaker regularity conditions however, as in the case considered by Hartman, one of the relations can only be proved in one direction  thus providing a statement that involutivity is sufficient for unique integrability but leaving open the question of its necessity. We do mention however that some of the statements in \cite{H}, see for example the bottom of page 118, do not make it completely clear  if a double implication is suggested or which additional conditions would be required to obtain one.}.

% We will give below a simple example of a PDE for which the existence and uniqueness can be obtained by an application of our results and, as far as we know, not by any other methods.

\subsubsection{Dominated decompositions with linear growth} Continuous distributions arise naturally in Dynamical Systems as \( D\varphi \)-invariant distributions for some diffeomorphism \( \varphi: M \to M \). From a geometric point of view dynamical systems also provides us with a variety of interesting examples. Modifications of the Hertz-Hertz-Ures (\cite{HHU}) example lead to two dimensional distributions on $\mathbb{T}^3$ which are $C^1$ outside two toral fibers $\mathbb{T}^2 \times \{0\}$ and 
$\mathbb{T}^2 \times \{1\}$ and are non-uniquely integrable at these fibers. In fact this example even leads to more interesting phenomenon such as a continuous distribution which is not uniquely integrable but still admits a foliation. Therefore dynamical systems play an important role in study of continuous distribtions as a provider of interesting examples. On the other direction the integrability (or not) of such distributions can have significant implications for ergodic and topological properties of the dynamics generated by \( \varphi \). The classical Frobenius Theorem and its various extensions have generally not been suitable for studying the integrability of such  ``dynamically defined'' distributions which are usually given implicitly by asymptotic properties of the dynamics and therefore have low regularity. The conditions we give here, on the other hand, are naturally suited to treat these kind of distributions  because they allow distributions with low regularity and also because they formulate the notion of involutivity in an asymptotic way which lends itself to be verified by sequences of dynamically defined approximations to the invariant distributions.

A first non-trivial application of the results of this paper is given in \cite{TurWar} for a class of   \( C^{2} \) diffeomorphisms \( \varphi: M \to M \) of a \( 3 \)-dimensional manifold which admit a \emph{dominated splitting}: there exists a continuous \( D\varphi \)-invariant tangent bundle decomposition \( TM=E\oplus F \) and a Riemannian metric for which derivative restricted to the \( 1 \)-dimensional distribution \( F \) is \emph{uniformly expanding}, i.e. \( \|D\varphi_{x}\|>1 \) for all \( x\in M \), and the derivative  restricted to the \( 2 \)-dimensional distribution \( E \) may have a mixture of contracting, neutral, or expanding behaviour but is in any case \emph{dominated} by the derivative restricted to \( F \), i.e. \( \|D\varphi_{x}(v)\|<\|D\varphi_{x}(w)\| \) for all \( x\in M \) and all unit vectors \( v\in E_{x}, w\in F_{x} \).

Dominated splittings, even on \( 3 \)-dimensional manifolds,  are not generally uniquely integrable \cite{HHU} but the main result of \cite{TurWar} is the unique integrability of dominated splittings on \( 3 \)-dimensional manifolds under the additional assumption that  the derivative restricted to \( E \) admits \emph{at most linear} growth, i.e. there exists a constant \( C>0 \) such that \( \|D\varphi^{k}_{x}v\|\leq Ck \) for all \( x\in M \), all unit vectors \( v\in E_{x} \), and all \( k\in \mathbb N \). This result is obtained by a non-trivial argument which leads to the verification that the distribution \( E \) is \emph{uniformly asymptotically involutive  on average} and therefore Theorem \ref{Frobenius-weak} can be applied, giving unique integrability.

Previous related results include unique integrability for splittings on the torus \( \mathbb T^{3} \) which admit a strong form of domination \cite{BriBurIva09} and  other results which assume various, rather restrictive, geometric and topological conditions \cite{Bri03, HaPo, Par, HeHeUr2}. The  assumption on linear growth is a natural extension of the most classical of all integrability results in the dynamical systems setting,  that of   \emph{Anosov diffeomorphisms}, where \( E \) is uniformly contracting, i.e.  \( \|D\varphi_{x}(v)\|<1<\|D\varphi_{x}(w)\| \) for all \( x\in M \) and for every unit vector \( v\in E_{x}, w\in F_{x} \). For Anosov diffeomorphisms the integrability of the invariant distributions can be obtained by a very powerful set of techniques which yield so-called ``Stable Manifold'' Theorems, which go back to Hadamard and Perron,  see \cite{HirPugShu77}. These techniques however generally break down  in settings where the domination is weaker.  The application of our Theorem \ref{Frobenius-weak}, as implemented in \cite{TurWar}, includes the setting of Anosov diffeomorphisms on \(  3  \)-dimensional manifolds and thus represents a perhaps more flexible, and maybe even more powerful in some respects, alternative to the standard/classical techniques.

\section{Strategy and main technical steps}\label{strategy}

Our approach is quite geometrical and implements the simple idea that if \(  \Delta  \) is a
 distribution  which is ``almost involutive'', it should be possible to apply a small perturbation to obtain a new
 distribution \(  \widetilde\Delta  \) which is involutive. It turns out that finding  \emph{some} perturbation to
 make the distribution involutive is easy, but making sure this perturbation is small  is non-trivial and essentially
 constitutes the key estimate in our argument.
 We will obtain  an estimate on the size of the perturbation which allows us to conclude that the asymptotic
 involutivity on average condition implies that \(  \Delta  \) can  be approximated by involutive  distributions. This
 can then be shown to imply (not necessarily unique) integrability of \(  \Delta  \).
Finally, we will use  an additional argument to show that we have unique integrability with the additional assumption of \emph{uniform} asymptotic involutivity on average.

In this section we reduce the proof of  Theorem \ref{Frobenius-weak} to the proof of some more technical statements, albeit also of independent interest. Before we proceed, however, we remark that we can assume without loss of generality that the forms \( \eta_k \) approximating \( \eta \) in the definition of asymptotic involutivity are actually \( C^2 \). Indeed, by a standard "mollification" procedure we can replace the original sequence with smoother ones which still satisfy the  (uniform) asymptotic involutivity on average conditions. Thus, from now on and for the rest of the paper we assume that the approximating forms \( \eta_k \) are \( C^2 \).

We will prove the following general perturbation result which does not require any involutivity or asymptotic involutivity assumptions. For two distributions \(  \Delta, \Delta'  \)  defined in some local chart \(  \mathcal U  \), we will use the notation \(  \measuredangle(\Delta,\Delta')  \) to denote the maximum angle between subspaces of \(  \Delta  \) and \(  \Delta'  \) at all points of \(  \mathcal U  \).

 \begin{thm}\label{solution}
Let $\Delta$ be a continuous 2-dimensional distribution on a 3-dimensional manifold $M$. Then, for every $x_0\in M$, there exist neighborhoods $\mathcal U'\subset \mathcal{U}$ of $x_0$ and  $\epsilon>0$  such that if $\Delta_{\epsilon}$ is a $C^2$ distribution with $\measuredangle(\Delta_\epsilon,\Delta)\leq \epsilon$ then there exists a local frame $\{X,Y\}$
 of $\Delta_{\epsilon}$ and a $C^1$ vector field $W$  such that the distribution
 \[
 \widetilde\Delta_{\epsilon} =span\{X+W,Y\}
 \]
  is involutive.  Moreover, \(  X, Y \) and \( W  \) can be chosen so that for every
 $C^2$   form \(  \eta  \) with $\Delta_{\epsilon}=\ker(\eta)$ and $\|\eta\|_{x}\geq 1$  for every \(  x\in \mathcal U
 \),  we have
 \[
 \|W\|\leq  \sup_{x\in\mathcal U, |t|\leq t_{0}} \{\|\eta\|_{x} \|\eta\wedge
 d\eta\|_{x}e^{\widetilde{d\eta}(x,t))}\}.
\] \end{thm} Theorem \ref{solution}  will be proved in Sections \ref{sec:pert}-\ref{pertbounds}. We remark that the condition \(  \|\eta\|_{x}\geq 1  \) is not a restriction  since the condition $\Delta_{\epsilon}=\ker(\eta)$ is preserved under multiplication of \(  \eta  \) by a scalar and therefore we can always assume without loss of generality that this lower bound holds. Its purpose is just to simplify the form of the upper bound on \(  \|W\| \) (where, as mentioned in the introduction, the norm \(  \|\cdot \|  \) refers to the \(  C^{0}  \) topology). Notice  that this bound is perfectly adapted to work with the asymptotic involutivity assumption of our main theorem. Indeed, by this assumption, for sufficiently large \(  k  \) we have that \(  \Delta^{(k)}  \) is close to \(  \Delta  \) and we can apply Theorem \ref{solution} to get a corresponding involutive distribution \(  \widetilde\Delta^{(k)}  \) after a perturbation  whose norm is bounded by \( \|\eta\| \|\eta\wedge d\eta\|e^{\widetilde{d\eta}(x,t))} \). Since \(  \eta_{k}\to \eta  \) we have that \(  \|\eta_{k}\|  \) is uniformly bounded, hence  \(   \|\eta\| \|\eta\wedge d\eta\|e^{\widetilde{d\eta}(x,t))}  \to 0  \)   and therefore the sequence of perturbed involutive distributions \(  \widetilde\Delta^{(k)}  \)  approximates the original distribution \(  \Delta  \). In Section \ref{sec:conv} we will show that this implies that  \(  \Delta  \) is (weakly) integrable in the sense that it admits (not necessarily unique) local integral surfaces through every point. We formalize this statement in the following

\begin{prop}\label{prop:int} Suppose there exists a sequence of involutive distributions $  \widetilde\Delta^{(k)}  $ which converges to a continuous distribution $\Delta $ uniformly on some open set  $\mathcal U$. Then there exists an open subset  $\mathcal V \subset \mathcal U$ such that $\Delta $ is (not necessarily uniquely) integrable at every $x \in \mathcal V$. \end{prop}

We note that the (local) convergence of a family of hyperplanes and line bundles is to be understood here in terms of the maximal angle going to zero uniformly in a  given neighbourhood. To get uniqueness of these integral manifolds  we will prove the following statement.

\begin{prop}\label{thm:uniqueness} Let $\Delta$ be a continuous 2-dimensional distribution on a 3-dimensional manifold $M$. Suppose that \( \Delta \) is  uniformly asymptotically involutive  on average. Then $\Delta$ is locally spanned by two  uniquely integrable vector fields \(  X, Y  \). \end{prop}

The unique integrability of $\Delta$ follows  from Proposition \ref{thm:uniqueness} by a simple contradiction argument: if there are two integral manifolds of $\Delta$ through a point then at least one of the vector fields
  $X$ and $Y$ does not satisfy uniqueness of solutions,  thus
   contradicting the statement of Proposition \ref{thm:uniqueness}.

We have thus reduced the proof of  Theorem \ref{Frobenius-weak} to the proofs of Theorem \ref{solution} which will be given in Sections \ref{sec:pert}-\ref{pertbounds}, Proposition \ref{prop:int} which will be given in Section \ref{sec:conv}, and Proposition \ref{thm:uniqueness}, which will be given in Section \ref{unicity}.

% % \begin{rem}We remark that the proof of the (weak) integrability (i.e. the \emph{existence} of integral manifolds, without necessarily having uniqueness) does not really require a uniform bound on \(  \|d\eta_{k}\|  \) as given in the definition of uniform asymptotic involutivity.  More precisely, in the upper bound we obtain on \(  \|W\|_{\infty}  \) in Theorem \ref{solution}, we could replace the term \( \|d\eta\|  \) by a term whose boundedness corresponds in some sense to \(  \|d\eta\|  \) being bounded in ``some direction'', this %  will become clear in Section \ref{pertbounds}. %  \end{rem}

\section{The perturbation}\label{sec:pert}

We now fix once and for all an arbitrary point \(  x_{0}\in M  \). Our aim in this section is to define a neighbourhood \(  \mathcal U  \) of \(  x_{0}  \) and a perturbation of a \(  C^{2}  \) distribution \(  \Delta_{\epsilon}  \) sufficiently close to our original distribution \(  \Delta  \) which yields a new \( C^{1} \) distribution \(  \widetilde\Delta_{\epsilon}  \). In the following sections we will show that \( \widetilde\Delta_{\epsilon}  \) satisfies the required properties for the conclusions of Theorem \ref{Frobenius-weak}, in particular that it is involutive and that it is a \emph{small} perturbation of \( \Delta_{\epsilon} \).

First of all we fix a local chart
 $(x^1,x^2,x^3, \mathcal{U}_0)$   centered at $x_0$. Notice that we can (and do) assume without loss of generality
 that \(  \Delta  \) is everywhere transversal to the coordinate axes in \(  \mathcal U_{0}  \) and that therefore
 this transversality also holds for \(  \Delta_{\epsilon}  \) if \(  \epsilon  \) is sufficiently small. In particular
 this implies that we can define a local frame
 \( \{X, Y\} \)
for \(  \Delta_{\epsilon}  \) in \(  \mathcal U_{0}  \) where \(  X, Y  \) are  \( C^2 \) vector fields of the form
 \begin{equation}\label{X-def}
 X=\frac{\partial}{\partial x^1}+a\frac{\partial}{\partial x^{3}}
 \quad \text{ and } \quad
 Y=\frac{\partial}{\partial x^2}+b\frac{\partial}{\partial x^{3}}.
 \end{equation}
for suitable \( C^2 \) functions  \(  a(x), b(x)  \). Notice  that the transversality condition implies that the \(  C^{0}  \) norms of \(  a  \) and \(  b  \) are uniformly bounded   below for all \(  \Delta_{\epsilon}  \) with \(  \epsilon  \) sufficiently small.

\begin{rem} The vector fields \(  X, Y  \) are \(  C^{2}  \) and thus define local flows, which we will denote by \(  e^{tX}, e^{tY}  \) respectively,  and admit unique integral curves through every point \(  x \in \mathcal U_{0} \), which we denote by \(  \mathcal X_{x}, \mathcal Y_{x}  \) respectively.  These integral curves will play an important role in the following construction and it will be sometimes convenient to mix the notation a little bit. For example we will refer to the ``natural'' parametrization of an integral curve \(  \mathcal Y_{x}  \) to intend parametrization by the flow so that once we specify some point \(  y=\mathcal Y_{x}(0)  \) (which may be different from the point \(  x  \) which we use to specify the curve) we then have \(  \mathcal Y_{x}(t) := e^{tY}(y) \).

\end{rem}

We are now ready to fix the neighbourhood \(  \mathcal U  \) in which we define the perturbation. At this stage we make certain choices motivated by the fact that the involutive distribution we are constructing is given by the span of two vector fields of the form \(  \{X+W, Y\}  \). We could similarly obtain  a pair of vector fields of the form \(  \{X, Y+Z\}  \) since the situation is completely symmetric. We let \(  \mathcal S \) denote the integral manifold through \(  x_{0}  \) of the coordinate planes given by \(  <\partial/\partial x^{1}, \partial/\partial x^{3}>  \) in the local chart \(  \mathcal U_{0}  \).  Then the vector field \(  Y  \) and its unique integral curves  are everywhere transversal to \(  \mathcal S  \) and indeed, by the uniform bounds on \(  |b(x)|  \), this transversality is uniform in \(  \Delta_{\epsilon}  \) as long as \(  \epsilon  \) is sufficiently small. In particular this means that we can choose a smaller neighbourhood \(  \mathcal U \subset \mathcal U_{0} \) which is ``saturated'' by the integral curves of \(  Y  \) in the sense that every point \(  x\in \mathcal U  \) lies on an integral curve of \(  Y  \) through some point of \(  \mathcal S\cap \mathcal U  \). Moreover, this saturation condition can be guaranteed for a fixed neighbourhood \(  \mathcal U  \) for any \(  \Delta_{\epsilon}  \) sufficiently close to \(  \Delta  \). For every \( x\in \mathcal U \) we let  \( \mathcal Y_x \) denote the integral curve through \( x \) of the vector field \( Y \). We consider the natural parametrization of each integral curve \( \mathcal Y_x \) by fixing the initial condition $\mathcal{Y}_x(0)\in \mathcal S$ and then let  $t_x$ be the time
 such that $\mathcal{Y}_x(t_x)=x$.  Notice that by choosing our neighbourhood \(  \mathcal U  \) sufficiently small,
 we can also assume that the integration time \(  t_{x}  \) is bounded by any a priori given arbitrarily small
 constant.
To simplify the final expression it will be convenient to have
 \begin{equation}\label{t}
 |t_{x}|\leq \frac{1}{\|X\|\cdot \|Y\|} = \frac{1}{\sqrt{(1+a^{2})(1+b^{2})}}\leq 1.
 \end{equation}
This upper bound is uniform for all distributions \(  \Delta_{\epsilon}  \) is \(  \epsilon  \) if sufficiently small. % %We first want to construct a neighborhood $\mathcal{U}'\subset\mathcal{U}$ %and a surface $S$ such that every point in $\mathcal{U}'$ can be joined to $S$ by % integral curve of $Y$ for all $Y$ coming from a distribution whose angle with % $\Delta_0$ is less than $\theta_0$. Notice that the construction of such surface % and neighborhood $\mathcal{U}'$ is equivalent to find a cross section of $Y$ in the %  the neighborhood $\mathcal{U}'$. To do so we need to understand the geometric interpretaion %  of the flow of $Y$, and by the explicit formula of $Y$ one sees that the flow %  of $Y$ is a motion in along the second and third axes of the coordinate system %   as we describe in the following lemma. % % %\begin{lem}\label{varioref} % For every $x\in\mathcal{U}$ and $t\in\mathbb{R}$ such that $e^{tY}(x)\in\mathcal{U}$ we have % $$e^{tY}(x)=e^{t\frac{\partial}{\partial x^2}}\circ P^t(x)$$ % where $P^t$ is the flow given by the vector field $ b\circ e^{-t\frac{\partial}{\partial x^2}}.\frac{\partial}{\partial x^3}$ %\end{lem} %\begin{proof} % Let $x\in\mathcal{U}$ and $t\in\mathbb{R}$ such that $e^{tY}(x)\in\mathcal{U}$, we denote % $P^t(x)=e^{-t\frac{\partial}{\partial x^2}}\circ e^{tY}(x)$ then by chain rule we have % $$\frac{dP^t}{dt}=-\frac{\partial}{\partial x^3}+e^{-t\frac{\partial}{\partial x^2}}_*Y.$$ % Since $Y=\frac{\partial}{\partial x^2}+b\frac{\partial}{\partial x^3}$ then we have % $$\frac{dP^t}{dt}=b\circ e^{-t\frac{\partial}{\partial x^2}}.$$ %\end{proof} % %The formula in Lemma \ref{varioref} is called variational formuala see \cite{A} for details. %This expression of the flow of $Y$ says that if we flow a submanifold tangent %to $\{\frac{\partial }{\partial x^1},\frac{\partial }{\partial x^3}\}$ is the composition %of two motions: %\begin{itemize} % \item shift in the horizontal direction $\{\frac{\partial }{\partial x^1},\frac{\partial }{\partial x^3}\}$, %\item lift in the vertival direction $\frac{\partial }{\partial x^2}.$ % \end{itemize} %By the variational formula the shift in the horizontal direction is at most %$t\sup|b|$ where $t$ the time of the flow. % %Then clearly we can chose a surface $S$ tangent to $\{\frac{\partial }{\partial x^1},\frac{\partial }{\partial x^3}\}$ %and the neighborhood $\mathcal{U}'$ that depend only on $sup|b|$. And more generally %by setting %$$\overline{b}:=sup\{b(x), x\in\mathcal{U}, \measuredangle(\Delta_0,\Delta)<\epsilon_0\}$$ %we can find a surface $S$ tangent to $\{\frac{\partial }{\partial x^1},\frac{\partial }{\partial x^3}\}$ %and the neighborhood $\mathcal{U}_{\overline{b}}$   such that every point in $\mathcal{U}_{\overline{b}}$ %can be joined to $S$ by integral curve of $Y$ for all $Y$ coming from a distribution whose angle with % $\Delta_0$ is less than $\theta_0$. %For the rest we consider $\mathcal{U}=\mathcal{U}_{\overline{b}}.$ %

We are now ready to define our perturbation. Notice first that the explicit forms of the vector fields \(  X  \) and \(  Y  \) implies that the Lie bracket \( [X, Y] \) always lies in the \(  \partial/\partial x^{3}  \) direction. Indeed, we can compute explicitly the Lie bracket and use it to define a function \( h: \mathcal U \to \mathbb R \) by
 \begin{equation}\label{h}
 [X,Y]=\left(\frac{\partial b}{\partial x^{1}} - \frac{\partial a}{\partial x^{2}} + a \frac{\partial b}{\partial
 x^{3}}-b\frac{\partial a}{\partial x^{3}}\right) \frac{\partial}{\partial x^{3}} =: h \frac{\partial}{\partial x^3}.
\end{equation} Thus \(  h  \) is  the (signed) magnitude of the Lie bracket \( [X,Y]   \) (which happens to be always in the  \(  \partial/\partial x^{3}  \) direction). We define the function  $\alpha:\mathcal{U}\to\mathbb{R}$  by \begin{equation}\label{alpha}
 \alpha(x):=\int_0^{t_x}h(\mathcal{Y}_x(\tau))\exp\left(\int_{\tau}^{t_x}\frac{\partial b}{\partial
 x^3}(\mathcal{Y}_x(s))ds\right)d\tau.
\end{equation} At the moment the function \( \alpha \) is just defined ``out of the blue'' with no immediately obvious motivation, but we will show below that it is exactly the right form for the perturbation we seek. Using this function we define the perturbed distribution by
 \begin{equation}\label{deltaepsilon}
 \widetilde\Delta_{\epsilon} :=span\left\{X+\alpha\frac{\partial }{\partial x^3},Y\right\}
 \end{equation}
In  Section \ref{sec:diff} we will show that \(  \widetilde\Delta_\epsilon  \) is \(  C^{1}  \), in Section \ref{sec:invol} that it is involutive, and in Section \ref{pertbounds} we will show that the perturbation \(  \alpha  \) satisfies the required upper bounds.

\section{Differentiability}\label{sec:diff} In this section we prove the following

\begin{prop}\label{C1}
 The function $\alpha$ is $C^1$.
\end{prop} Since  \( X \) and \( Y \) are \( C^2 \), it follows immediately from Proposition \ref{C1} and  \eqref{deltaepsilon} that  \(  \widetilde\Delta_\epsilon  \) is \( C^{1} \) as required.

To prove Proposition \ref{C1}, notice first that from the definition of \( \alpha \) in \eqref{alpha} it follows immediately that
 $\alpha$ is $C^1$ in the direction of $Y$. It is therefore sufficient to prove that \( \alpha \) is also \( C^1 \)
 along
  two other vector fields that together with \( Y \) form a coordinate system in $\mathcal{U}$.
  The existence of such a coordinate system in \( \mathcal U \) is guaranteed by classical results on the
  representation of vector fields near a regular point, see e.g. \cite{Lee}, however we will need here a particular
  choice of coordinate system,  in particular one defined by \( Y \) and two additional vector fields \( Z, V \) which
  span the tangent space of \( \mathcal S \).  Therefore we give a self contained proof.

\begin{prop}\label{cano} There are two $C^1$ vector fields $Z$ and $V$ that span the tangent space of $\mathcal S$, such that the system  $\{Y, V,Z\}$ is a trivialization of the tangent bundle $TM$ by commuting vector fields in the neighborhood $\mathcal{U}$, which is to say that $$[Y,Z]=[Y,W]=[Z,W]=0.$$ \end{prop}

\begin{proof} We  first recall that the neighborhood $\mathcal{U}$ is parametrized in such a way that any point can be joined to a point of $\mathcal S$ by an integral curve of $Y$, and so we can choose $\epsilon>0$, and modify \( \mathcal U \) slightly, such that the  map
 $$\phi(t_1,t_2,t_3)=e^{t_1Y}\circ e^{t_2\frac{\partial}{\partial x^1}}\circ e^{t_3\frac{\partial}{\partial
 x^3}}(x_0)$$
 is a diffeomorphism from  $(-\epsilon,\epsilon)^{3}$ to $\mathcal{U}$.
We define $$ Z:= \phi_*\frac{\partial}{\partial t_2} \quad\text{ and } \quad V:=\phi_*\frac{\partial}{\partial t_3} $$ where the subscript \( * \), here and below, denotes the standard push-forward of vector fields. Observe that by the chain rule, for every \(  \bar t=(t_{1}, t_{2}, t_{3})\in (-\epsilon, \epsilon)^{3}  \) we have \[ \phi_*\frac{\partial}{\partial t_2}(\phi(\bar t))=\frac{\partial\phi}{\partial t_2}(\bar t)= e^{t_1Y} _*\frac{\partial }{\partial x^1}(\phi(\bar t)) \] and \[ \phi_*\frac{\partial}{\partial t_3}(\phi(\bar t))=\frac{\partial\phi}{\partial t_3} (\bar t)= e^{t_1Y}_*\frac{\partial }{\partial x^3}(\phi(\bar t)). \] Since the vector field $Y$ is  $C^2$ it follows that  $e^{tY}_*$ is $C^1$ which implies that the vector fields $Z$ and $V$ are $C^1$. By the naturality of the Lie bracket and observing that $Y= \phi_*\frac{\partial}{\partial t_1}$ we have $$ [V,Y]=\left[\phi_*\frac{\partial}{\partial t_3},\phi_*\frac{\partial}{\partial t_1}\right]= \phi_*\left[\frac{\partial}{\partial t_3},\frac{\partial}{\partial t_1}\right]=0 $$ and similarly $[Z,Y]=[Z,V]=0$. This shows that the vector fields commute.  Now we are only left to prove that $\mathcal S$ is spanned by $V$ and $Z$. Since \( \mathcal S \) is by definition the integral surface of the local coordinates \( \partial/\partial x^1 \) and \( \partial/\partial x^3 \), it is sufficient to show that \( V \) and \( Z \) span this plane. We will show this by computing explicit formulas for the vector fields. By standard calculus for vector fields on manifolds \cite{AgrSac04, Gro96}, for any \(  x\in \mathcal U  \), we have that $$ \frac{d}{dt}\left(e^{tY}_*\frac{\partial}{\partial x^3}|_x\right)=e^{tY}_*\left[\frac{\partial}{\partial x^3},Y\right]|_{x}= \frac{\partial b}{\partial x^3}\circ e^{-tY}(x) \cdot e^{tY}_*\frac{\partial}{\partial x^3}|_x. $$ Integrating  both side gives $$e^{tY}_*\frac{\partial}{\partial x^3}|_x=\exp\left(\int_0^{t}\frac{\partial b}{\partial x^3}\circ e^{-\tau Y}(x)d\tau\right)\frac{\partial}{\partial x^3}|_x $$ which shows that \( V \) always lies in the direction \( \partial/\partial x^3 \). By the same calculations we also get $$ e^{tY}_*\frac{\partial}{\partial x^1}|_x=\frac{\partial}{\partial x^1}|_x+\int_0^t\frac{\partial b}{\partial x^1}\circ e^{-sY}(x) \exp\left(\int_{s}^{t}\frac{\partial b}{\partial x^3}\circ e^{-\tau Y}(x)d\tau\right)ds\frac{\partial}{\partial x^3}|_x $$ which shows that \( Z \) always lies in the span of \( \partial/\partial x^1 \) and \( \partial/\partial x^3 \). Therefore we have that \( V \) and \( Z  \) span the tangent space of \( \mathcal S \). \end{proof}

To complete the proof of Proposition \ref{C1} it is sufficient to show that \( \alpha \) is \( C^1 \) along the vector fields \( Z \) and \( V \) defined above. We will need the following simple fact which constitutes the main motivation for our specific choice of the coordinate system.

\begin{lem}\label{time} If  $x$ and $y$ belong to the same integral curve of $V$ or to the same integral curve of $Z$, then  we have
 $$t_x=t_y.$$
\end{lem} \begin{proof}
 Let $x$ and $y$ be in the same integral curve of $V$. Then, since $[Y,V]=0$, it follows that  $e^{t_xY}(x)$ and
 $e^{t_xY}(y)$ are
 in the same integral curve of $V$. Since $e^{t_xY}(x)\in \mathcal S$ and $V\in T\mathcal S$ then $e^{t_xY}(y)\in
 \mathcal S$ and it follows that
 $t_x=t_y.$ The proof for $Z$ is exactly the same.
\end{proof}

\begin{proof}[Proof of Proposition \ref{C1}] To show that  $\alpha$ is differentiable along the vector fields  $Z$ and $V$ we will show directly from first principle that   for every $x\in\mathcal{U}$, the limits $$\lim_{\delta\to0}\frac{\alpha(x)-\alpha(e^{\delta V}(x))}{\delta} \quad \text{ and } \quad \lim_{\delta\to0} \frac{\alpha(x)-\alpha(e^{\delta Z}(x))}{\delta} $$
 exist. We will prove the statement for the first limit, the second follows by exactly the same arguments. We fix some
 $x\in\mathcal{U}$ and for $\delta\neq 0$,  by Lemma \ref{time} we have $t_x=t_{e^{\delta V}(x)}=:t$.
 To simplify the notation in the calculations below, we shall write
\[ \mathcal B(\delta, \tau) := \exp\left(\int_{\tau}^{t}\frac{\partial b}{\partial x^3}(\mathcal{Y}_{e^{\delta V}(x)}(s))ds\right). \] Notice that for \( \delta=0 \) we have \( e^{\delta V}(x)=x \) and so
 we have
\[ \alpha(x) = \int_0^{t}h(\mathcal{Y}_x(\tau))\mathcal B(0, \tau) d\tau. \] Then $$ \alpha(x)-\alpha(e^{\delta V}(x))=\int_0^{t}h(\mathcal{Y}_x(\tau)) \mathcal B(0, \tau)d\tau -\int_0^{t}h(\mathcal{Y}_{e^{\delta V}(x)}(\tau))\mathcal B(\delta, \tau)d\tau. $$ By adding and subtracting the term $$ \int_0^t h(\mathcal{Y}_x(\tau)) \mathcal B(\delta, \tau) d\tau $$ to the right hand side we get \[ \begin{aligned} \alpha(x)-\alpha(e^{\delta V}(x)) =\int_0^{t}&h(\mathcal{Y}_x(\tau)) \mathcal B(0, \tau)d\tau -\int_0^t h(\mathcal{Y}_x(\tau)) \mathcal B(\delta, \tau) d\tau \\ &+\int_0^t h(\mathcal{Y}_x(\tau)) \mathcal B(\delta, \tau) d\tau -\int_0^{t}h(\mathcal{Y}_{e^{\delta V}(x)}(\tau))\mathcal B(\delta, \tau)d\tau \\ =\int_0^{t}&h(\mathcal{Y}_x(\tau)) [\mathcal B(0, \tau) -  \mathcal B(\delta, \tau)] d\tau \\ &+\int_0^t [h(\mathcal{Y}_x(\tau)) - h(\mathcal{Y}_{e^{\delta V}(x)}(\tau))]\mathcal B(\delta, \tau)d\tau. \end{aligned} \] Dividing both sides  by \( \delta \) it is therefore sufficient to show that the limit exists for each integral on the last two lines above. For the first integral notice that \( h(\mathcal Y_x(\tau)) \) does not depend on \( \delta \) and therefore it is sufficient to show that \[ \lim_{\delta\to 0}\frac{\mathcal B(0, \tau) -  \mathcal B(\delta, \tau)}{\delta} \] exists.  To see this, notice first that it is equal to \[\lim_{\delta\to 0}\frac{1}{\delta}\left[\exp\left(\int_{\tau}^{t}\frac{\partial b}{\partial x^3}(\mathcal{Y}_x(s))ds\right)-\exp\left(\int_{\tau}^{t}\frac{\partial b}{\partial x^3}(\mathcal{Y}_{e^{\delta V}(x)}(s))ds\right)\right], \] This  is by definition the directional derivative of the function \begin{equation}\label{eq:dirder}
 \exp \int_\tau^t \frac{\partial b}{\partial x^3} (\mathcal Y_x(s))ds
 \end{equation}
in the direction of \( V \). Since \( {\partial b}/{\partial x^3} \) is \( C^1 \) it follows that \eqref{eq:dirder} is also \( C^1 \) and therefore this directional derivative exists. Similarly, for the second integral above, the limit of \( \mathcal B(\delta, \tau)\) as \( \delta\to 0 \) is just \( \mathcal B(0, \tau) \) (and thus exists), and so it is sufficient to show that the limit \[ \lim_{\delta\to 0}\frac{h(\mathcal{Y}_x(\tau)) - h(\mathcal{Y}_{e^{\delta V}(x)}(\tau))}{\delta}. \] exists. Again, this is exactly the directional derivative of \( h \) in the direction of \( V \). Since \( h \) is \( C^1 \) this derivative exists. This proves that  $\alpha$ is $C^1.$
 \end{proof}

\section{Involutivity}\label{sec:invol}

In this Section we prove \[ \left[X+\alpha\frac{\partial }{\partial x^3},Y\right]=0. \] This implies involutivity as required, since the vanishing of the Lie bracket for a \(  C^{1}  \) local frame of a \(  C^{1}  \) distribution is well known to be equivalent to the involutivity condition \(  \eta\wedge d\eta=0  \) given above; this follows for example from Cartan's formula given in \eqref{invariant} below, or see any standard reference such as \cite{Lee}. By the linearity of the Lie bracket we have \[ \left[X+\alpha\frac{\partial }{\partial x^3},Y\right] = [X,Y] + \left[\alpha\frac{\partial }{\partial x^3},Y\right] \] and, applying  the general formula \(  [\varphi X, \psi Y]= \varphi X(\psi)Y- \psi Y(\varphi)X+\varphi\psi [X,Y]  \) for \(  C^{1}  \) functions \(  \varphi, \psi \), where \(  X(\psi), Y(\varphi)  \) denote the ``directional derivatives'' of the functions \(  \psi, \varphi  \) in the directions of the vector fields \(  X, Y  \) respectively, we  get \[ \left[\alpha\frac{\partial }{\partial x^3},Y\right] =-Y(\alpha)\frac{\partial }{\partial x^3}+\alpha\frac{\partial b}{\partial x^3}\frac{\partial }{\partial x^3} \] Notice that this bracket lies in the \(  x^{3}  \) direction. Substituting  above  and using the fact that \(  [X,Y] = h \partial/\partial x^{3}  \) also lies in the \(  x^{3}  \) direction we get \begin{equation}\label{eq:inv} \left[X+\alpha\frac{\partial }{\partial x^3},Y\right] =\left(h+\alpha\frac{\partial b}{\partial x^3}-Y(\alpha)\right)\frac{\partial }{\partial x^3}. \end{equation} Thus the involutivity of \(  \Delta_{\epsilon}  \) is equivalent to the condition that  the bracket on the left hand side of \eqref{eq:inv} is equal to  0, or equivalently that \(  Y(\alpha) = h+\alpha {\partial b}/{\partial x^3}  \), i.e.
  that \(  \alpha  \) is a solution to the partial differential equation
\begin{equation}\label{eq:inv2} Y(u)=h+ u\frac{\partial b}{\partial x^3}. \end{equation} To see  that \(  \alpha  \) is a solution of \eqref{eq:inv2} note that by the definition of \(  \alpha  \) in \eqref{alpha}, for any integral curve \(  \mathcal Y  \) of \(  Y  \) in \(  \mathcal U  \) parametrized so that \(  \mathcal Y(0)\in \mathcal S  \) and for any \(  |t|\leq t_{0}  \)  we have \[ \alpha(\mathcal{Y}(t)) = \int_0^{t}h(\mathcal{Y}(\tau))\exp\left(\int_{\tau}^{t}\frac{\partial b}{\partial x^3}(\mathcal{Y}(s))ds\right)d\tau. \] Differentiating \(  \alpha  \) along \(  \mathcal Y  \) we get \[ \begin{aligned} Y(\alpha)(\mathcal{Y}(t))  = &\frac{d}{dt}\alpha(Y(t))  \\ =& \ h(\mathcal{Y}(t))\exp\left(\int_{t}^{t}\frac{\partial b}{\partial x^3}(\mathcal{Y}(s))ds\right) \\ +& \frac{\partial b}{\partial x^3}(\mathcal{Y}(t))\int_0^{t}h(\mathcal{Y}(\tau))\exp\left(\int_{\tau}^{t}\frac{\partial b}{\partial x^3}(\mathcal{Y}(s))ds\right)d\tau \\ =& \ h(\mathcal{Y}(t)) + \frac{\partial b}{\partial x^3}(\mathcal{Y}(t)) \alpha(\mathcal{Y}(t)) \end{aligned} \] which proves that $\alpha$ is the required solution of \eqref{eq:inv2}.

\section{Perturbation bounds}\label{pertbounds}

In this Section we prove the upper bound on the norm of \(  \alpha  \) which gives the upper bound required in the statement of   Theorem \ref{solution}. Notice first of all that by the definition of the function \(  h  \) in \eqref{h} we have \begin{equation}\label{alpha1} \begin{aligned} \alpha(x) &=\int_0^{t_x}h(\mathcal{Y}_x(\tau))\exp\left(\int_{\tau}^{t_x}\frac{\partial b}{\partial x^3}(\mathcal{Y}_x(s))ds\right)d\tau
 \\ &\leq
 \int_0^{t_x}\|[X,Y](\mathcal{Y}_x(\tau))\|\exp\left(\int_{\tau}^{t_x}\frac{\partial b}{\partial
 x^3}(\mathcal{Y}_x(s))ds\right)d\tau
\end{aligned} \end{equation} We will estimate the two terms in two Lemmas.

\begin{lem}\label{bound1} \(  \|[X,Y](\mathcal{Y}_x(\tau))\| \leq  \{\|X\|_{\mathcal{Y}_x(\tau)} \|Y\|_{\mathcal{Y}_x(\tau)} \|\eta\wedge d\eta\|_{\mathcal{Y}_x(\tau)} \} \) \end{lem} \begin{lem}\label{bound2} \( \exp\left(\int_{\tau}^{t_x}\frac{\partial b}{\partial x^3}(\mathcal{Y}_x(s))ds\right)
 \leq\|\eta\|_{\mathcal{Y}_x(\tau)}\exp(\widetilde{d\eta}(\mathcal{Y}_x(\tau), t_{x}-\tau)
\) \end{lem} Combining these two estimates, substituting into \eqref{alpha1}, and using  the bound on \(  t  \) given by \eqref{t}, we obtain \begin{align*} |\alpha(x)| &\leq t \sup_{x\in\mathcal U} \{\|\eta\|_{x}\exp (\widetilde{d\eta}(t,x))  \|X\|_{x} \|Y\|_{x} \|\eta\wedge d\eta\|_{x} \} \\ &\leq \sup_{x\in\mathcal U}\{ \|\eta\|_{x}\|\eta\wedge d\eta\|_{x}\exp(\widetilde{d\eta}(t,x))\} \end{align*} which is the required bound and thus completes the proof of Theorem \ref{solution} modulo the proof of the two Lemmas. For the proof of  both Lemmas,  notice first  that, since the vector fields \(  X, Y  \) defined in \eqref{X-def}  lie in   \(  \ker(\eta)  \), any \(  C^{1}  \) form \(  \eta  \)   such that  $\Delta=ker(\eta)$ is of the form \begin{equation}\label{eta} \eta=c(dx^3-adx^1-bdx^2) \end{equation} for some non-vanishing \(  C^{1}  \)  function \(  c(x)  \) defined in \(  \mathcal U  \). Notice that \(  \eta(\partial/\partial x^{3}) = c  \) and therefore \(  c\leq \|\eta\|  \) everywhere and  we can
 even assume, up to multiplying \(  \eta  \) by a (possibly negative) scalar if necessary,  that
\begin{equation}\label{c}
 1\leq c\leq\|\eta\|
\end{equation} We can now prove the two Lemmas. \begin{proof}[Proof of Lemma \ref{bound1}] All the estimates below are made for a given fixed point in \(  \mathcal U  \) and so for simplicity we omit this from the notation. By the definition of \(  h  \) in \eqref{h} we have \(  \eta([X, Y]) = h \eta (\partial/\partial x^{3}) = ch  \) and therefore \( h =  {\eta ([X, Y])}/{c} \) and in particular \[ \|[X, Y]\| = \frac{|\eta([X,Y])|}{c} \] Since \(  X, Y \in \ker (\eta) \),  we have \(  \eta(X)=\eta(Y)=0  \) and the ``Cartan formula'' gives \begin{equation}\label{invariant}
 d\eta(X,Y)=X(\eta(Y))-Y(\eta(X))-\eta([X,Y])= -\eta([X,Y]).
\end{equation} On the other hand, we have \[ \eta\wedge d\eta \left(\frac{\partial}{\partial x^{3}}, X, Y\right) = \eta \left(\frac{\partial}{\partial x^{3}}\right) d\eta(X,Y)= c d\eta (X,Y). \] Substituting into the equations above we then get \[ \|[X, Y]\| = \frac{|\eta([X,Y])|}{c} = \frac{|d\eta(X,Y)|}{c}= \frac{1}{c^{2}} \left|\eta\wedge d\eta \left(\frac{\partial}{\partial x^{3}}, X, Y\right)\right| \] Using that \(  c>1  \) and the multilinearity of \(  \eta\wedge d\eta  \) this gives the bound \( \|[X, Y]\| \leq \|X \|\|Y\| \|\eta\wedge d\eta\| \) as required. \end{proof}

\begin{proof}[Proof of Lemma \ref{bound2}] By direct calculation we have
 $$
 \begin{aligned}
 d\eta&=(\frac{\partial c}{\partial x^1}+a\frac{\partial c}{\partial x^3}+c\frac{\partial a}{\partial x^3})dx^1\wedge
 dx^3
 +(\frac{\partial c}{\partial x^2}+b\frac{\partial c}{\partial x^3}+c\frac{\partial b}{\partial x^3})dx^2\wedge
 dx^3\\
 &+(a\frac{\partial c}{\partial x^2}-b\frac{\partial c}{\partial x^1}+c\frac{\partial a}{\partial x^2}-c\frac{\partial
 b}{\partial x^1})dx^1\wedge dx^2\\
 &=(X(c)+c\frac{\partial a}{\partial x^3})dx^1\wedge dx^3+(Y(c)+c\frac{\partial b}{\partial x^3})dx^2\wedge dx^3\\
 &+(a\frac{\partial c}{\partial x^2}-b\frac{\partial c}{\partial x^1}+c\frac{\partial a}{\partial x^2}-c\frac{\partial
 b}{\partial x^1})dx^1\wedge dx^2
\end{aligned}
 $$
On the other hand we can write $$ d\eta=d\eta_1dx^1\wedge dx^3+d\eta_2dx^2\wedge dx^3+d\eta_3dx^1\wedge dx^2 $$ and so,  by comparing the terms of the two formulae for $d\eta$, we have $$ Y(c)=-c\frac{\partial b}{\partial x^3}+d\eta_2. $$ Dividing both sides by $c$ gives $$\frac{Y(c)}{c}=-\frac{\partial b}{\partial x^3}+\frac{d\eta_2}{c}.$$ Since \(  Y(c)  \) is exactly the derivative of \(  c  \) along integral curves of \(  Y  \),  integrating along these integral curves we get $$ \log\left|\frac{c(\mathcal{Y}_{x}(t_{x}))}{c(\mathcal{Y}_{x}(\tau))}\right|=-\int_\tau^{t_{x}}\frac{\partial b}{\partial x^3}(\mathcal{Y}_{x}(s))ds+ \int_\tau^{t_{x}}\frac{d\eta_2(\mathcal{Y}_{x}(s))}{c(\mathcal{Y}_{x}(s))}ds $$ which implies $$ \int_\tau^{t_{x}}\frac{\partial b}{\partial x^3}(\mathcal{Y}_{x}(s))ds = \int_\tau^{t_{x}}\frac{d\eta_2(\mathcal{Y}_{x}(s))}{c(\mathcal{Y}_{x}(s))}ds -\log\left|\frac{c(\mathcal{Y}_{x}({t_{x}}))}{c(\mathcal{Y}_{x}(\tau))}\right| $$ hence we have $$\exp\left(\int_\tau^{t_{x}}\frac{\partial b}{\partial x^3}(\mathcal{Y}_{x}(s))d\tau\right)= \left|\frac{c(\mathcal{Y}_{x}(\tau))}{c(\mathcal{Y}_{x}({t_{x}}))}\right| \exp\left(\int_\tau^{t_{x}}\frac{d\eta_2(\mathcal{Y}_{x}(s))}{c(\mathcal{Y}_{x}(s))}ds\right). $$ Using   that  \(  1\leq c \leq \|\eta\|  \) by \eqref{c} we then get \[ \exp\left(\int_\tau^{t_{x}}\frac{\partial b}{\partial x^3}(\mathcal{Y}_{x}(s))d\tau\right) \leq \|\eta\|_{\mathcal{Y}_{x}(\tau)}\exp\left(\int_\tau^{t_{x}}d\eta_2(\mathcal{Y}_{x}(s)) ds\right). \] Notice that the integral on the right hand side is not exactly in the form used in the definition of \(  \widetilde{d\eta}_{2}  \) in Section \ref{sec:ave}, where the limits in the integral go from 0 to \(  t_{x}  \). Recalling that $\mathcal{Y}_{x}(s)=e^{(s-t_x)Y}(x)$, we have \[ \int_\tau^{t_{x}}d\eta_2(\mathcal{Y}_{x}(s)) ds= \int_\tau^{t_{x}}d\eta_2 \circ e^{(s-t_x)Y}(x) ds= \int_0^{t_{x}-\tau}d\eta_2\circ e^{(s-t_x)Y}(e^{\tau Y}(x))) ds \] This last integral is by definition equal to \(  \widetilde {d\eta_{2}}(\mathcal Y_{x}(\tau), t_{x}-\tau)  \) and so substituting into the expression above this completes the proof. \end{proof}

\section{Convergence}\label{sec:conv} In this section we prove Proposition \ref{prop:int}. We suppose throughout  that we have a sequence \(  \tilde\Delta^{(k)}  \) of \(  C^{1}  \) involutive distributions converging uniformly to a continuous distribution \(  \Delta  \) in some open set \(  \mathcal U  \). The involutivity of the distributions \( \widetilde{\Delta}^{(k)} \) implies that they are uniquely integrable by the classical Frobenius Theorem, but to prove the required convergence we will need to construct these integral manifolds rather explicitly.

We assume without loss of generality that the open set \(  \mathcal U  \) is contained inside some local chart and that   $\Delta$ is everywhere transversal to the coordinate axes in this local chart. By the convergence of the sequence of distributions \(  \tilde\Delta^{(k)}  \) to \(  \Delta  \), the same transversality property holds for \(  \tilde\Delta^{(k)}  \) for all sufficiently large \(  k  \). This  implies that each \(  \tilde\Delta^{(k)}  \) admits a local frame \(  \{X_{k}, Y_{k}\}  \) formed by \(  C^{1}  \) vector fields of the form
 \begin{equation}\label{X-defk}
 X_{k}=\frac{\partial}{\partial x^1}+a_{k}\frac{\partial}{\partial x^{3}}
 \quad \text{ and } \quad
 Y_{k}=\frac{\partial}{\partial x^2}+b_{k}\frac{\partial}{\partial x^{3}}.
 \end{equation}
 for \(  C^{1}  \) functions \(  a_{k}, b_{k}  \).
By the convergence of the sequence of distributions \(  \tilde\Delta^{(k)}  \) to \(  \Delta  \) it follows that the sequences of vector fields \(  X_{k}, Y_{k}  \) converge to continuous vector fields \(  X, Y  \) which form a continuous local frame of \(  \Delta   \)
 and have the form
  \begin{equation}\label{X-deflim}
 X=\frac{\partial}{\partial x^1}+a\frac{\partial}{\partial x^{3}}
 \quad \text{ and } \quad
 Y=\frac{\partial}{\partial x^2}+b\frac{\partial}{\partial x^{3}}.
 \end{equation}
 for continuous functions \(  a, b  \) (cf. \eqref{X-def}).
 Since the approximating vector fields \(  X_{k}, Y_{k}  \) are \(  C^{1}  \), their Lie bracket is well defined and
 their specific form implies it lies in the \(  \partial/\partial x^{3}  \) direction, see \eqref{h}, and in
 particular is transversal to \(  \tilde\Delta^{(k)}  \). Therefore by the involutivity of \(  \tilde\Delta^{(k)}  \)
 it follows that the vector fields commute, i.e.
 \begin{equation}\label{commute}
 [X_{k}, Y_{k}]=0.
 \end{equation}
Notice that of course we cannot draw the same conclusion for the vector fields \(  X, Y  \) since they are only continuous and the Lie bracket is not defined.

We now fix an open subset \(  \mathcal V \subset \mathcal U  \) and some sufficiently small \(  \epsilon>0  \) such that   for any \( x\in \mathcal V \), any  \( |s_1|, |s_2| \leq \epsilon \) and all sufficiently large \(  k  \), we have \[ W^{(k)}_{x}(s_1,s_2) := e^{s_1X_k}\circ e^{s_2 Y_k}(x) \in \mathcal U. \] Notice that  by \eqref{commute} we have that \begin{equation}\label{tang} \begin{aligned} \frac{\partial W^{(k)}_x}{\partial s_1}(s_1,s_2) &= X_k(W^{(k)}_x(s_1,s_2)) \\ \frac{\partial W^{(k)}_x}{\partial s_2}(s_1,s_2) &= Y_k(W^{(k)}_x(s_1,s_2)). \end{aligned} \end{equation} In particular the Jacobian of $W^{(k)}_x$ is non-zero everywhere on $[-\epsilon, \epsilon]^2$ and therefore we have a sequence of embeddings \[ W^{(k)}_x: [-\epsilon, \epsilon]^2 \rightarrow \mathcal U. \] By \eqref{tang}, the tangent spaces of \( W^{(k)}_x \) are exactly the hyperplanes of the distributions  \( \widetilde\Delta^{(k)} \) and therefore the images of the maps \( W^{(k)}_x \) are exactly the local
 integral manifolds of the distributions through the point \( x \).

\begin{lem} For every \(  x\in \mathcal V  \),  sequence  $\{W^{(k)}_x\}$ is equicontinuous and equibounded \end{lem}

\begin{proof} We have that $W^{(k)}_x(0)=x$ so $W^{(k)}_x(s_1,s_2) \in  U$ and therefore equiboundedness is easy. For equicontinuity note that $DW^{(k)}(s_1,s_2)$ is a matrix whose columns are $X_k( W^{(k)}(s_1,s_2)$ and $Y_k( W^{(k)}(s_1,s_2)$. Therefore the differential is equibounded and so $W^{(k)}_x(s_1,s_2)$ is equicontinuous. \end{proof}

 By the Arzela-Ascoli Theorem
% \footnote{ %\begin{thm}Let $U\subset \mathbb{R}^n$ a compact subset and $(E,|\cdot|)$ a finite dimensional Banach space over $\mathbb{R}$. If $\mathcal{F}$ is a family of equicontinuous functions from $U$ to $E$ and $\{f_j\}\subset \mathcal{F}$ is an equibounded sequence then $\{f_j\}$ contains a convergent subsequence which converges uniformly to a continuous function. %\end{thm} %} there exists   a continuous function \[ W_x: [-\epsilon,\epsilon]^2 \rightarrow \mathcal U. \]
 which is the uniform limit of some subsequence of $\{W^{(k)}_x\}_{k=1}^{\infty}$(which we assume, without loss of
 generality, to be the full sequence from now on). To complete the proof of Proposition \ref{prop:int} it is therefore
 sufficient to show that \( W_x([-\epsilon, \epsilon]^2)  \) is an integral manifold of the limiting distribution \(
 \Delta \), i.e. that \( W_x \) is actually differentiable and that its tangent spaces coincide with the hyperplanes
 of \( \Delta \). Thus, letting
 $DW_x=DW_x(s_1,s_2)$ denote the matrix whose columns are are $X(W_x(s_1,s_2))$ and
 $Y(W_x(s_1,s_2))$ it is sufficient to prove the following

\begin{lem} $W_x$ is a differentiable function whose derivative is $DW$ \end{lem} \begin{proof} Notice first of all that from \eqref{tang} and the fact that \( X_k\to X \), \( Y_k\to Y \) and $W^{(k)}_x(s_1,s_2) \to W_x(s_1,s_2)$ it follows that the partial derivatives \( {\partial W^{(k)}_x}/{\partial s_i} (s_1,s_2)\) converge to \( X(W_x(s_1,s_2)) \) and \( Y(W_x(s_1,s_2)) \) respectively, and therefore, the derivative \( DW^{(k)}_x \) converges uniformly to \( DW_x \). Now, for any two points $p,q \in W_x([-\epsilon,\epsilon]^2)$ and a smooth curve $\gamma=\gamma(t)$ connecting $p$ to $q$ with \( \gamma(0)=p, \gamma(\tau)=q \),  the fundamental theorem of calculus implies that \[ W^{(k)}_x(q) = W^{(k)}_x(p) + \int_{0}^{\tau} DW^k_x(\gamma(t)) \circ \frac{d\gamma(t)}{dt}dt \] Thus, taking limits and exchanging the limit and the integral (which can be done due to uniform convergence) we get $$ W_x(q) = W_x(p) + \int_{0}^{\tau} DW_x(\gamma(t))\circ \frac{d\gamma(t)}{dt}dt $$ This implies that  $DW_x$ is the derivative of $W_x$ and thus proves in particular that \( W_x \) is differentiable as required. \end{proof}

\section{Uniqueness}\label{unicity} 

In this section we first complete the proof of Theorem \ref{Frobenius-weak} in the case of rank-1 distributions on surfaces. For clarity we restate the notation and the result in this simplest setting as Proposition \ref{uniq} below. We will then use this result to complete the proof of  Proposition \ref{thm:uniqueness} and thus complete the proof of Theorem \ref{Frobenius-weak} for rank 2 distributions on 3-dimensional manifolds.

 We consider a Riemannian surface \( \mathcal S \) and a vector field \( X \) on \( \mathcal S \) which we can suppose
 to be  given as  $X=ker(w)$ for some continuous \( 1 \)-form \( w \) defined on  $\mathcal{S}$. We can restrict our
 attention to a local chart \( \mathcal U \) with local coordinates  $(z^1,z^2,\mathcal{U})$ and suppose without loss
 of generality that \( X \) is everywhere transversal to both  coordinate axes \( \partial/\partial z^{1},
 \partial/\partial z^{2} \) in \( \mathcal U \) and that therefore in particular it can be written in the form
 \[
 X=\frac{\partial}{\partial z^1}+a \frac{\partial}{\partial z^2}
 \]
 for some non-zero continuous function \( a(z) \). If \(  w_{k}  \) is a sequence of \(  C^{1}  \) \(  1  \)-forms on
 \(  \mathcal U  \) with \(  w_{k}\to w  \) then, for all sufficiently large \(  k  \), the corresponding vector
 fields \(  X_{k}  \) will also be transversal to both axes and the corresponding vector fields \(  X_{k}=ker (w_{k})
 \) can also be written in the form
  \[
 X_{k}=\frac{\partial}{\partial z^1}+ a_{k}\frac{\partial}{\partial z^2}
 \]
 for \(  C^{1}  \) functions \(  a_{k}  \). Choosing some smaller domain \(  \mathcal U'\subset \mathcal U  \) there
 exists some \(  t_{0}>0  \) such that the flow \(  e^{\tau X_{k}}  \)  is well defined for all \(  x\in \mathcal U'
 \) and \(  |\tau|\leq t_{0}  \) and \(  e^{\tau X_{k}}(x)\in \mathcal U  \). Notice that in this case the external
 derivatives \(  dw_{k}  \) of the \(  1  \)-forms \(  w_{k}  \) have only one component and so, by some slight abuse
 of notation we can simply write \(  dw_{k}=dw_{k}dz^{1}\wedge dz^{2}  \). Then,  for every \(  z\in\mathcal U'   \)
 and every \(  |t|\leq t_{0}  \) define
 \[
\widetilde{dw}_{k}(z,t):=\int_{0}^{t}dw_{k}\circ e^{\tau X_{k}}(z) d\tau . \]
 With this notation we then have the analogue of Theorem \ref{Frobenius-weak} as follows.

\begin{prop} \label{uniq} Let \( X=ker (w) \) be a continuous vector field defined on a surface \(  \mathcal S  \). Suppose that for every point there is a local chart \( \mathcal U \) and a sequence of $C^1$ differential 1-forms $w_k$
 on \( \mathcal U \) such that \(  w_{k}\to w  \), the corresponding \(  C^{1}  \) vector fields \(  X_{k}  \), and a
 neighbourhood \(  \mathcal U'\subset \mathcal U  \) such that for every \(  z\in \mathcal U'  \) and every \(
 |t|\leq t_{0}  \)
 \begin{equation}\label{uniq2d}
 \|w_k-w\|_{z} e^{\widetilde{dw}_{k}(z,t)}\to 0
 \end{equation}
as \(  k\to \infty  \). Then $X$ is uniquely integrable. \end{prop}

It might be interesting to know  if this result also admits an analogue for vector fields on \( \mathbb R \) giving conditions for uniqueness in this most simple setting, and thus allowing comparison with existing results such as \cite{Bre88}. The argument for the proof that we give below does not admit an immediate ``restriction'' to the one-dimensional setting. We first show how it implies  Proposition \ref{thm:uniqueness}.

\begin{proof}[Proof of Proposition \ref{thm:uniqueness} assuming Proposition \ref{uniq}] Let $\Delta=\ker(\eta)$ be a  uniformly asymptotically involutive on average continuous distributions as per the hypotheses of Proposition \ref{thm:uniqueness}. Then, by definition, for every \(  x_{0}\in M  \) there exists  local coordinates \(  (x^{1}, x^{2}, x^{3})   \) in some local chart \(  \mathcal U  \) in which we have vector fields \(  X_{k}, Y_{k}, X, Y  \) as in \eqref{X-defk} and \eqref{X-deflim}. We just need to show that \( X \) and \( Y \) are uniquely integrable. We will prove  unique integrability for \( X \) using Proposition \ref{uniq}, the argument for \( Y \) is completely analogous.

 Notice first of all that  \(  X \)  is contained in the surface tangent to the coordinate axes $<\partial/{\partial
 x^1}, {\partial}/{\partial x^3}>$.
Moreover, as in \eqref{eta} above,  the explicit form of the vector fields \( X_{k} \) and \( X \) mean that the forms \(  \eta_{k}  \) and \(  \eta  \) can be written as \[
 \eta_k=c_k(dx^3-a_kdx^1-b_kdx^2) \quad\text{ and } \quad
  \eta = c (dx^3-adx^1-bdx^2)
 \]
 for some non-vanishing \(  C^{1}  \) function \(  c_{k}(x)  \) and   continuous function \(  c(x)  \) respectively,
 and that
their restriction to the surface \(  \mathcal S  \) locally tangent to $<\partial/{\partial x^1}, {\partial}/{\partial x^3}>$ yields the forms \[ w_k=c_k(dx^3-a_kdx^1)\quad\text{ and } \quad
  w = c (dx^3-adx^1)
\] with the property that \(  X_{k}=\ker(w_{k})  \) and \(  X=\ker(w)  \). We therefore just need to show that the forms \(  w_{k}, w  \) satisfy the assumptions of Proposition \ref{uniq} to get unique integrability of the vector field \(  X  \). The convergence is immediate since the assumption that \(  \eta_{k}\to   \eta  \) implies that \(  a_{k}\to a  \), \(  b_{k}\to b  \) and \(  c_{k}\to c  \) and therefore in particular that \(  w_{k}\to w  \) as \(  k\to\infty  \). To show that \(  \|w_k-w \|_{x}e^{\widetilde{dw}_{k}(x,t)}\to0  \) we have, by direct calculation, \[ \begin{aligned} d\eta_k&=\left(X_k(c_k)+c_k\frac{\partial a_k}{\partial x^3}\right)dx^1\wedge dx^3+\left(Y_k(c_k)+c_k\frac{\partial b_k}{\partial x^3}\right)dx^2\wedge dx^3\\ &+\left(a\frac{\partial c_k}{\partial x^2}-b_k\frac{\partial c_k}{\partial x^1}+c_k\frac{\partial a_k}{\partial x^2}-c\frac{\partial b_k}{\partial x^1}\right)dx^1\wedge dx^2 \end{aligned} \] and $$ dw_k=\left(X_k(c_k)+c_k\frac{\partial a_k}{\partial x^3}\right)dx^1\wedge dx^3. $$ Therefore \( dw_{k}=d\eta_{k,1} \) and thus \eqref{uniq2d} follows immediately from the assumption that \( \eta \) is uniformly asymptotically involutive on average. \end{proof}

The rest of this section is  devoted to the proof Proposition \ref{uniq}. 
The main technical step is contained in Lemma \ref{unik} below
which is a sort of one-dimensional version of Theorem \ref{solution} where  we showed that each distribution can be perturbed to yield an involutive distribution and that the size of this perturbation can be controlled. Here we show that any 1-form \(  \eta  \) defining a vector field on a surface can be ``rescaled''  to a \emph{closed} 1-form defining \ \emph{the same vector field} and that this rescaling  has controlled norm.

\begin{lem}\label{unik} Let $w$ be \(  C^{2}  \) differential $1$-form on a surface \(  \mathcal S  \) and $(z^1,z^2,\mathcal{U})$ be a coordinate systems whose axes are transverse to $ker(w)$ and let $\mathcal{L}$ be an integral curve of $\partial/\partial z^2$.
 Then there is  a \(  C^{1}  \)
differential 1-form $\hat w$ with $\ker(\hat w)=\ker(w)$ and $\hat w({\partial}/{\partial z^2})=1$ along $\mathcal{L}$ such that for every \(  z\in \mathcal U  \) we have $$d\hat w=0 \quad\text{ and } \quad \|\hat w\|_{z}\leq \sup_{|t|\leq t_{0}}\{e^{\widetilde{dw}(z,t))}\} \|w\|_{z}. $$ \end{lem}

\begin{proof} The proof also proceeds along quite similar lines to the proof of Theorem \ref{solution}, though the situation here is considerably simpler. By the transversality of $ker (w)$ to the axes we have that
    $w=c(dz^2-bdz^1)$
  for some $C^2$ functions $b,c$ and, without loss of generality, we
  assume that $1\leq|c|\leq\|\eta\|$.
 We assume that the neighborhood $\mathcal{U}$ is parametrized
such that every point $z\in\mathcal{U}$ corresponds to a time $t_z$ such that the integral curve $\mathcal{X}$ of $X:= {\partial}/{\partial x^1}+b {\partial}/{\partial x^2}$ is so that $ e^{-t_zX}(z)\in\mathcal{L}$. For every \(  z\in \mathcal U  \) let
 $$
 \beta(z):=\exp\left(-\int_0^{t_z}\frac{\partial b}{\partial x^2}(e^{(\tau-t_z)X}(z))d\tau\right).
 $$
We can now define the form $$ \hat w:=\beta(dz^2-bdz^1). $$ Then, by definition of $\hat w$ we have $\ker(\hat w)=\ker(w)$ and for $z\in\mathcal{L}$ we have  $t_z=0$ which implies that the integral above vanishes and so \(  \beta(z)=1  \) and so  $\hat w(\partial/\partial z^2)=1$. Therefore we just need to show that \(  \hat w  \) is \(  C^{1}  \) and satisfies the required bounds. From the form of \(  w  \) and \(  \hat w  \) we have \begin{equation}\label{beta} \|\hat w\|_{z}\leq |\beta (z)|  \|dz^2-bdz^1\|_{z} = |\beta (z)|\frac{\|w\|_{z}}{c(z)} = \frac{\beta(z)}{c(z)}\|w\|_{z}. \end{equation} It is enough therefore to bound \(  \beta(z)/c(z)  \). Notice first that, as can be verified in a straightforward way,  the function $\beta$ is the unique solution of the partial differential equation
 $$X(u)=-u\frac{\partial b}{\partial z^2}
 $$
with boundary conditions \(  u=1  \) on \(  \mathcal L  \). By the same arguments as in the proof of Proposition \ref{C1}, we have that  $\beta$ is $C^1$ and, by direct calculation, $$ d\hat w =\left(X(\beta)+\beta\frac{\partial b}{\partial z^2}\right)dz^1\wedge dz^2=0. $$ Thus \(  \hat w \) is closed. Finally, to estimate the norm of \(  \hat w  \), again by direct calculation, we have $$ dw=\left(X(c)+c\frac{\partial b}{\partial z^2}\right)dz^1\wedge dz^2. $$ By a slight abuse of notation again we write \( dw = X(c)+c {\partial b}/{\partial z^2}. \) Then, dividing through by \(  c  \), we have $$ \frac{X(c)}{c}+\frac{\partial b}{\partial z^2}=\frac{dw}{c}. $$ Hence for every $z\in\mathcal{U}$ and $\tau\in [0,t_z] $ we have $$ \frac{X(c)}{c}\circ e^{(\tau-t_z)X}(z)+\frac{\partial b}{\partial z^2}\circ e^{(\tau-t_z)X}(z)=\frac{dw}{c}\circ\circ e^{(\tau-t_z)X}(z). $$

Integrating this equality along an integral curve of $X$ we get $$ \log\left|\frac{c(z)}{c\circ e^{-t_zX}(z))}\right|-\int_0^{t_z}\frac{dw}{c}\circ e^{(\tau-t_z)X}(z)d\tau =-\int_0^{t_z}\frac{\partial b}{\partial x^2}\circ e^{(\tau-t_z)X}(z)d\tau $$ and hence, taking exponentials, $$ \begin{aligned} \exp\left(-\int_0^{t_z}\frac{\partial b}{\partial x^2}\circ e^{(\tau-t_z)X}(z)d\tau\right)&= \left|\frac{c(z)}{c\circ e^{-t_zX}(z))}\right| \exp\left(-\int_0^{t_z}\frac{dw}{c}\circ e^{(\tau-t_z)X}(z)d\tau\right)\\ &=\left|\frac{c(z)}{c\circ e^{-t_zX}(z))}\right|\exp\left(\int_0^{-t_z}\frac{dw}{c}\circ e^{\tau X}(z)d\tau\right) \end{aligned} $$ Finally, using that $1\leq|c|\leq\|\eta\|$ , this gives \[ \beta(z) \leq c(z) \exp\left(-\int_0^{t_z}\frac{dw}{c}\circ e^{(\tau-t_z)X}(z)d\tau\right) \leq c(z) \exp (\widetilde{dw}(z,-t_z)). \] Substituting into \eqref{beta} gives the result. \end{proof}

\begin{proof}[Proof of Proposition \ref{uniq}] We suppose by contradiction that the vector field  $X$ admits two integral curves $\mathcal{X}^1$ and $\mathcal{X}^2$  through a given point $z_0\in \mathcal S$ parametrized so that \(  \mathcal X^{1}(0)=\mathcal X^{2}(0)=z_{0}  \). Now let $(z^1, z^2, \mathcal{U})$  be the local coordinate system around \(  z_{0}  \) given by the assumptions of the Theorem. In particular  $X$ is transverse to both coordinate axes, and in particular to \(  \partial/\partial z^2  \) and therefore there exists an  integral curve  $\mathcal{L}$ of $\partial/\partial x^2$ which joins two points
 $z_1=\mathcal{X}^{1}(s_1)$ and $z_{2}=\mathcal{X}^{2}(s_2)$ for some $s_1,s_2$.
We can suppose that $\mathcal{L}$ is parametrized such that $\mathcal{L}(0)=z_1$ and $\mathcal{L}(t_2)=z_2$. Let $\Gamma$ be the closed curve given by union of $\mathcal{L}$ and the two integral curves of $X$ through $z_0$, let \(  D  \) be the region bounded by \(  \Gamma  \),  and let $\hat w_k$ be the sequence of $1$-forms given by Lemma \ref{unik}. By Stokes' formula we have \begin{equation}\label{cont}
 \int_{\Gamma}\hat w_k=\int_{D}d\hat w_k=0
\end{equation}
 and therefore
$$ \int_{\Gamma}\hat w_k=\int_{\mathcal{X}^1}\hat w_k+ \int_{\mathcal{X}^2}\hat w_k+\int_{\mathcal{L}}\hat w_k=0 $$ and so \begin{equation}\label{b1}
 \left|\int_{\mathcal{X}^1}\hat w_k+
\int_{\mathcal{X}^2}\hat w_k\right|= \left|\int_{\mathcal{L}}\hat w_k\right|. \end{equation} By Lemma \ref{unik} we have   \(  \hat w_{k}(\partial/\partial z^{2})=1\) along \( \mathcal L   \) and therefore the right hand side of \eqref{b1} is equal to \(  |t_{2}|  \) where \(  t_{2}  \) is the ``distance'' between \(  z_{1}  \) and \(  z_{2}  \) along \(  \mathcal L  \). By assumption \(  t_{2}\neq 0  \) (and is independent of \(  k  \)) but we will show that the left hand side of \eqref{b1} tends to 0 as \(  k\to \infty  \), thus giving rise to a contradiction as required.

To estimate the left hand side of \eqref{b1}, notice that since the curves $\mathcal{X}^i$ are tangent to $X$ we have $$ \int_{\mathcal{X}^i}\hat w_k=\int_0^{s_i}\hat w_k (X)(\mathcal{X}^i(t))dt $$ Also since $\hat w_k(X_k)=0$ (since \(  X_{k}=ker(\hat w_{k})  \)) we can write \[ \int_{\mathcal{X}^i}\hat w_k =\int_0^{s_i}\hat w_k (X-X_k)(\mathcal{X}^i(t))dt \] Now let \(  |t_{i}|\leq s_{i}  \) be such that \[ \int_0^{s_i}\hat w_k (X-X_k)(\mathcal{X}^i(t))dt = s_{i} \hat w_k (X-X_k)(\mathcal{X}^i(t_{i})). \] Then, letting \(  y=\mathcal X^{i}(t_{i})  \) we have \begin{equation}\label{b2} \begin{aligned} \int_{\mathcal{X}^i}\hat w_k &=\int_0^{s_i}\hat w_k (X-X_k)(\mathcal{X}^i(t))dt \\ &\leq s_i\|\hat w_k\|_y\|X-X_k\|_y \leq s_i\|\hat w_k\|_y\|w-w_k\|_y \end{aligned} \end{equation} By Lemma \ref{unik} we have \[ \|\hat w_k\|_y \leq \sup_{|t|\leq t_{0}}\{e^{\widetilde{dw}(y,t))}\} \|w\|_{y} \] Substituting this into \eqref{b2} and applying the assumptions of the Theorem we get that \(  \int_{\mathcal{X}^i}\hat w_k \to 0  \) for \(  i=1,2  \) and, as explained above, this leads to a contradiction and thus completes the proof. \end{proof}

\section{Applications} In this section we prove Theorem \ref{thm-ODE} on the uniqueness of solutions for ODE's, and Theorem \ref{pfaff} on the existence and uniqueness of solutions for a class of Pfaff PDE's.

\subsection{Uniqueness of solutions for continuous ODE's}\label{ODEproof} To prove Theorem \ref{thm-ODE} we fix once and for all some reference point \( x_{0}\in \mathcal U \) and choose a sufficiently small ball \( B(x_{0},r)\subset \mathcal U \). We will show that the vector field \( X \) is uniquely integrable at each point \( x\in B(x_{0}, r) \). Notice first of all that, since \( f(x)=(f_{1}(x), f_{2}(x)) \) is non-vanishing in \( \mathcal U \), we can assume, without loss of generality and passing to a smaller radius \( r \) if necessary, that there exist local coordinates where $f_1(x)>0$ for all $x\in B(x_{0}, r)$. Then we can rescale the vector field \( X \) by dividing through by \( f_{1} \) and define \[ Y := \frac{\partial}{\partial x_1} + g(x)\frac{\partial}{\partial x_2} \quad \text{ where } \quad g(x) = \frac{f_2}{f_1}(x) \] Note that \( g \) has the same modulus of continuity $\omega(t)$ as \( f \) and unique integrability of $Y$ is equivalent to that of $X$ since the two vector fields are just rescalings one of the other (and thus define the same one-dimensional distributions). It is therefore sufficient to show that \( Y \) is uniquely integrable. To do this we first show that we can approximate the function \( g \) be a family of smooth functions \( g^{\epsilon} \) satisfying certain approximation bounds, and thus approximate the continuous vector field \(  Y  \) by corresponding smooth vector fields \(  Y^{\epsilon}  \). We then apply Proposition \ref{prop:approx} to get unique integrability of \(  Y  \).

For any small \(  \epsilon>0  \) let  \( V_{x_{0}, \epsilon} \) denote the Riemannian volume of the ball \( B(x_{0},\epsilon) \). Then,  letting \( w(t) \) denote the modulus of continuity of the function \( g \) above, we have the following result.

\begin{prop}\label{prop:approx} There exists a family of smooth functions \( \{g^{\epsilon}\}_{\epsilon > 0} \) defined on \(  B(x_{0}, r)  \) and a constant $K>0$ such that for every \( \epsilon > 0 \) we have \[ |g^{\epsilon} - g|_{\infty} \leq \frac{K}{\epsilon}\int_{0}^{\epsilon}\omega(t)dt \quad\text{ and } \quad \left|\frac{\partial g^{\epsilon}}{\partial x}\right| \leq \frac{K}{\epsilon^{2}}\int_{0}^{\epsilon} \omega(t)dt \] \end{prop}

\begin{proof}[Proof of Theorem \ref{thm-ODE} assuming Proposition \ref{prop:approx}] Notice first of all that  \( Y = ker (w) \), where  \( w \) is the continuous 1-form \( w := dx-g(x)dt. \) Then there exists a family of approximating vector fields \(  Y^{\epsilon}  \) given by the smooth  \( 1 \)-forms \( w^{\epsilon}:=dx-g^{\epsilon}(x)dt \) where the smooth functions \( g^{\epsilon} \) are given by Proposition~\ref{prop:approx}. Notice that since \(  w^{\epsilon}  \) are smooth they admit exterior derivatives \(  dw^{\epsilon}  \) and \(  |dw^{\epsilon}|_{\infty}= |\partial g/\partial x|_{\infty}  \). Then from  Proposition~\ref{prop:approx}  we get 
\[
 |w^{\epsilon}-w|_{\infty}e^{\epsilon |dw^{\epsilon}|_{\infty}} \leq \frac{K}{\epsilon}\int_{0}^{\epsilon} \omega(t)dt \exp\left( \frac{K}{\epsilon}\int_{0}^{\epsilon} \omega(t)dt \right) 
\] 
and therefore,  since the right hand side tends to zero by assumption, we have
 \(
|w^{\epsilon}-w|_{\infty}e^{\epsilon |dw^{\epsilon}|_{\infty}}  \to 0 \) as \(  \epsilon \to 0  \). 
A direct application of Theorem \ref{Frobenius} gives the  unique integrability of \( Y \). 
\end{proof}

\begin{proof}[Proof of Proposition \ref{prop:approx}] We will use some fairly  standard approximations by mollifiers. In particular the calculations below follow closely the ones in \cite{S2} but are formulated in terms of modulus of continuity rather than the H\"older norm. Let \(  r\in (0,1)  \) and $\phi$ be at the standard  mollifier supported on \(  B(x_{0}, r/2)  \) and, for every \(  \epsilon\in (0,1) \), let
 \[
 \phi_{\epsilon} := \frac{1}{V_{x_{0}, \epsilon r/2}}\phi \left(\frac{x}{\epsilon}\right).
\quad\text{ and } \quad g^{\epsilon}(x) := \int_{B(x_{0},\epsilon)}\phi_{\epsilon}(y)g(x-y)dy. \]
 By the well-known properties of mollifiers we have  that \(  \phi_{\epsilon}  \) is supported on \(  B(x_{0},
 \epsilon)  \), \(  \int \phi_{\epsilon}=1,  |\phi_{\epsilon}|_{\infty}\leq 1/V_{x_{0}, \epsilon r/2}  \),
 \[
\int_{B(x_{0},\epsilon)}\frac{\partial \phi_{\epsilon}}{\partial x}=0, \quad\text{ and }\quad \frac{\partial \phi_{\epsilon}}{\partial x}(x) = \frac{1}{\epsilon V_{x_{0}, \epsilon r/2}}\frac{\partial \phi}{\partial x}(\frac{x}{\epsilon}). \] Therefore, using these properties \[ |g^{\epsilon}(x)-g(x)| \leq  \int_{B(0,\epsilon)}|\phi_{\epsilon}(y)| |g(x-y)-g(x)|dy \leq |\phi_{\epsilon}|_{\infty}\int_{B(x_{0},\epsilon)}\omega(|y|)dy. \] Passing  to polar coordinates $(r,\theta)$ and noting that $|y| = r$ and that the volume form in polar coordinates has the form $dV = r  drd\theta$, we get \[ \begin{aligned} |g^{\epsilon}(x)-g(x)| &\leq |\phi_{\epsilon}|_{\infty}\int_{B(x_{0},\epsilon)}\omega(|y|)dy \leq \frac{2\pi}{V_{x_{0}, \epsilon r/2}} \int_{0}^{\epsilon}t\omega(t)dt \\ &\lesssim \frac 1{\epsilon^{2}}\int_{0}^{\epsilon}t\omega(t)dt \leq \frac 1\epsilon \int_{0}^{\epsilon}\omega(t)dt \end{aligned} \] which proves the first claim. Similarly \[ \begin{aligned} \left|\frac{\partial g^{\epsilon}}{\partial x}(x)\right| &=\left|\int_{B(x_{0},\epsilon)}\frac{\partial \phi_{\epsilon}}{\partial x}(y)g(x-y)dy \right|\\ &=|\int_{B(x_{0},\epsilon)}\frac{\partial \phi_{\epsilon}}{\partial x}(y)(g(x-y)-g(x))dy|\\ &\leq \int_{B(x_{0},\epsilon)}\left |\frac{\partial \phi_{\epsilon}}{\partial x}(y)\right|\omega(|y|)dy\\ &=|d\phi_{\epsilon}|_{\infty}\int_{B(x_{0},\epsilon)}|\omega(|y|)dy \\ &=|d\phi|_{\infty}\frac{2\pi}{\epsilon V_{x_{0}, \epsilon r/2}}\int_{0}^{\epsilon}t\omega(t)dt \\ &\lesssim \frac 1{\epsilon^{3}}\int_{0}^{\epsilon}t\omega(t)dt \leq \frac 1{\epsilon^{2}} \int_{0}^{\epsilon}\omega(t)dt \end{aligned} \] where last line is again achieved by passing to polar coordinates This completes the proof of the proposition. \end{proof}

\subsection{Pfaff's Problem}\label{sec:pfaff} The proof of Theorem \ref{pfaff} is based on the observation that
 notion of integrability for the PDE \eqref{Fro}
 is closely related to the geometric integrability of distributions in the following way.
Let  \(  \Delta  \) be the two-dimensional distribution in \(  \mathcal U  \) spanned by the local frame  \(   \{ X, Y\}  \) with \begin{equation}\label{eq:span}
 \quad X:=\frac{\partial}{\partial x}+a\frac{\partial}{\partial z}
 \quad \text{ and } \quad
 Y:=\frac{\partial}{\partial y}+b\frac{\partial}{\partial z}
\end{equation} where \(  a(x,y,z), b(x,y,z)  \) are  the functions in \eqref{Fro}.

\begin{prop}\label{pfaffint}
 \eqref{Fro} is (uniquely) integrable iff
 $\Delta$ is (uniquely) integrable.
 \end{prop}

\begin{proof} Suppose first that \eqref{Fro} is integrable. Let \(  (x_{0}, y_{0})\in \mathcal V  \) and let \(  f  \) be a solution of \eqref{Fro} with initial condition \(  f(x_{0}, y_{0})=z_{0}  \). Then the graph \[ \Gamma(f)=\{(x,y,f(x,y)), (x,y)\in \mathcal V\} \] is  an embedded surface in \(  \mathbb R^{3}  \) and the tangent space at a point \(  (x, y, f(x,y))  \)
  is spanned by
\begin{equation}\label{span} \frac{\partial}{\partial x}+\frac{\partial f}{\partial x}\frac{\partial}{\partial z} \quad\text{ and } \quad \frac{\partial}{\partial y}+\frac{\partial f}{\partial y}\frac{\partial}{\partial z}. \end{equation}
 Since \(  f  \) is a solution of \eqref{Fro}, then \eqref{span} is exactly of the form \eqref{eq:span}
and so the graph \(  \Gamma(f)  \) is a (unique)  integral manifold of  $\Delta.$ Conversely,  suppose that the distribution \(  \Delta  \) with local frame \eqref{eq:span} is (uniquely) integrable. Then the integral manifold of $\Delta$ through a point $(x_0,y_0,z_0)\in \mathcal U$ can be realized as a graph of a (unique) function $f=f(x,y)$. The tangent spaces of this graph are by definition given by the span of \eqref{eq:span} but also of \eqref{span} and thus  \(  f  \) is a solution of \eqref{Fro} which is therefore (uniquely) integrable. \end{proof}

\begin{proof}[Proof of Theorem \ref{pfaff}] We translate the problem into the geometric problem of the integrability of the corresponding distribution \(  \Delta  \) spanned by vector fields of the form \eqref{eq:span}.  Then  \(  \Delta  \) can be written as the kernel \(  \Delta=ker (\eta)  \) of the 1-form \[ \eta= dz - a dx - bdy. \] We will construct a sequence \(  \eta_{k}  \) of \(  C^{1}  \) with the property that \(  \eta_{k}\to \eta  \), \(  \|d\eta_{k}\|  \) uniformly bounded in \(  k  \), and \(  \|\eta_{k}\wedge d\eta_{k}\|\to 0  \).  This implies that \(  \Delta  \) is uniformly asymptotically involute and is thus uniquely integrable by Theorem \ref{Frobenius}. Let \(  F^{(k)}(z)  \) be a sequence of \(  C^{1}  \) functions such that \(  F^{(k)}\to F  \) in the \(  C^{0}  \) topology and such that the derivative \(  F_{z}^{(k)}  \) is uniformly bounded in \(  k  \) (this is possible because \(  F  \) is Lipschitz continuous). Now let \[ \eta_{k}:= dz - A^{(k)}F^{(k)}dx - B^{(k)}F^{(k)}dy. \] Then \(  \eta  \) is a \(  C^{1}  \) form and clearly \(  \eta_{k}\to \eta  \). Moreover, by direct calculation, we have \begin{align*} d\eta_{k} &= -d(A^{(k)}F^{(k)})\wedge dx - d(B^{(k)}F^{(k)})\wedge dy \\&= -A^{(k)}_{y}F^{(k)}dy\wedge dx - A^{(k)}F_{z}^{(k)}dz\wedge dx - B_{x}^{(k)}F^{(k)}dx\wedge dy - B^{(k)}F^{(k)}_{z} dz\wedge dy \\ &= (A^{(k)}_{y}-B_{x}^{(k)})F^{(k)} dx\wedge dy + A^{(k)}F_{z}^{(k)}dx\wedge dz + B^{(k)}F^{(k)}_{z} dy\wedge dz. \end{align*} Th functions \(  A^{(k)}, B^{(k)}, F^{(k)}  \) are converging to the corresponding functions \(  A, B, F  \) and are therefore uniformly bounded, the functions \(  F^{(k)}_{z}  \) are uniformly bounded by construction, and \(  A^{(k)}_{y}-B_{x}^{(k)} \to 0  \) by assumption. It follows that \(  \|d\eta_{k}\|  \) is uniformly bounded in \(  k  \). Finally, again by direct calculation we have \begin{align*} \eta_{k}\wedge d\eta_{k} =&
 [(A^{(k)}_{y}-B_{x}^{(k)})F^{(k)} - A^{(k)}F^{(k)}B^{(k)}F^{(k)}_{z} + B^{(k)}F^{(k)}A^{(k)}F^{(k)}_{z}]dx\wedge
 dy\wedge dz
 \\
 =& (A^{(k)}_{y}-B_{x}^{(k)})F^{(k)} dx\wedge dy\wedge dz.
\end{align*} Therefore \(  \|\eta_{k}\wedge d\eta_{k}\|\to 0  \) and this completes the proof. \end{proof}

\newpage Stefano Luzzatto \\ \textsc{Abdus Salam International Centre for Theoretical Physics (ICTP), Strada Costiera 11, Trieste, Italy}\\
 \textit{Email address:} \texttt{luzzatto@ictp.it}

\medskip Sina Tureli \\ \textsc{Abdus Salam International Centre for Theoretical Physics (ICTP), Strada Costiera 11, Trieste, Italy} and \textsc{International School for Advanced Studies (SISSA), Via Bonomea 265, Trieste}\\
 \textit{Email address:} \texttt{sinatureli@gmail.com}

\medskip Khadim  War\\ \textsc{Abdus Salam International Centre for Theoretical Physics (ICTP), Strada Costiera 11, Trieste, Italy} and \textsc{International School for Advanced Studies (SISSA), Via Bonomea 265, Trieste}\\
 \textit{Email address:} \texttt{kwar@ictp.it}

\end{document}